\documentclass[11pt]{article}
\usepackage{amsthm}
\usepackage{amsmath}
\usepackage{bbm}
\usepackage{epsfig}
\usepackage{graphicx}
\usepackage{mathrsfs}
\usepackage{amsfonts}
\usepackage{color}
\usepackage{enumerate}
\usepackage{rotating}
\usepackage{verbatim}
\usepackage{epstopdf}
\usepackage[authoryear,round,comma]{natbib}
\usepackage{hyperref}

\newtheorem{Thm}{Theorem}[section]

\newtheorem{Lem}[Thm]{Lemma}
\newtheorem{proposition}[Thm]{Proposition}
\newtheorem{theorem}[Thm]{Theorem}

\newtheorem{Cor}[Thm]{Corollary}

\newtheorem{Rem}[Thm]{Remark}
\newtheorem{example}[Thm]{Example}

\newtheorem{Ass}[Thm]{Assumption}

\newcommand{\EqD}{\overset{d}{=}}

\newcommand{\cl}{\mathcal}

\newcommand{\bb}{\mathbb}
\newcommand{\E }{\mathbb{E}}

\newcommand{\pr}{\mathbb{P}}

\begin{document}
\centerline{\large \bf Tail Adversarial Stability for Regularly Varying Linear Processes}
\centerline{\large \bf and their Extensions}
\centerline{Shuyang Bai\footnote{Corresponding author.}(\texttt{bsy9142@uga.edu}) and Ting Zhang (\texttt{tingzhang@uga.edu})}
\centerline{\em University of Georgia}
\centerline{\today}

\begin{abstract}
The notion of tail adversarial stability has been proven useful in obtaining limit theorems for tail dependent time series. Its implication and advantage over the classical strong mixing framework has been examined for max-linear processes, but not yet studied for additive linear processes. In this article, we fill this gap by verifying the tail adversarial stability condition for regularly varying additive linear processes.  We in addition consider extensions of the result to a stochastic volatility generalization and to a max-linear counterpart. We also address the invariance of tail adversarial stability under monotone transforms.   Some  implications for limit theorems in statistical context are also discussed.
\end{abstract}

\section{Introduction}\label{sec:introduction}
Compared with the conventional notion of correlation-based dependence that mainly concerns co-movements around the mean, tail dependence or extremal dependence refers to the dependence in the joint extremes that mainly concerns the co-occurrence of tail events. In bivariate or finite-dimensional distributions, the concept of tail dependence and its quantification have been extensively explored in the literature; see for example \citet{Sibuya:1960}, \citet{deHaan:Resnick:1977}, \citet{Joe:1993}, \citet{Ledford:Tawn:1996}, \citet{Coles:Heffernan:Tawn:1999}, \citet{Embrechts:Mcneil:Straumann:2002}, \citet{Draisma:Drees:Ferreira:DeHaan:2004}, \citet{Poon:Rockinger:Tawn:2004}, \citet{McNeil:Frey:Embrechts:2005}, \citet{Zhang:2008}, \citet{Balla:Ergen:Migueis:2014}, \citet{Hoga:2018}  and references therein. In the time series setting, \citet{Leadbetter:Lindgren:Rootzen:1983}, \citet{Smith:Weissman:1994} and \citet{Ferro:Segers:2003} considered the use of an extremal index to describe the degree of tail dependence. \citet{Zhang:2005} proposed to extend the bivariate tail dependence index of \citet{Sibuya:1960} to the time series setting, and \citet{Linton:Whang:2007} considered a variant that uses  a quantile as the tail threshold; see also the extremogram of \citet{Davis:Mikosch:2009} and \citet{Hill:2009}, the tail autocorrelation of \citet{Zhang:2022}, as well as  the notion of tail process in \citet{basrak2009regularly}    which characterizes the local tail dependence structure  in a stationary regularly varying time series.
 We also refer to a  recent review paper   \citet{ZhangZ:2021} and  monograph \citet{kulik2020heavy} for additional discussions and references.

Although various tail dependence measures have been proposed, most of them focused on summarizing the degree of tail dependence in   observed data and few is useful for developing limit theorems of tail dependent time series. This is similar to the correlation-based dependence case, in which the correlation coefficient is prominent in summarizing the underlying degree of dependence, but generally does not lead to sufficient conditions for limit theorems to be developed. To achieve the latter, a popular approach is due to the influential work of \citet{Rosenblatt:1956} which introduced the notion of strong mixing along with a ``big block-small block'' argument that have led to the development of various limit theorems under the strong mixing condition. Although not being originally developed for handling dependence in the tail, the strong mixing framework has been applied to the tail setting as a major tool for developing limit theorems; see for example \citet{Smith:Weissman:1994}, \citet{Drees:2003}, \citet{Ferro:Segers:2003}, \citet{Chernozhukov:2005}, \citet{Davis:Mikosch:2009}, \citet{Chernozhukov:FernandezVal:2011}, \citet{Davis:Mikosch:Cribben:2012}, \citet{Mikosch:Zhao:2014}, \cite{kulik2020heavy}, and references therein. To handle tail statistics from time series data, however, the strong mixing condition often has to be used together with additional anti-clustering conditions that control more specifically the degree of dependence in the tail; see for example condition (9.67) of \citet{Chernozhukov:2005}, condition (3.3) of \citet{Davis:Mikosch:2009}, Assumption 4 of \citet{Chernozhukov:FernandezVal:2011}, and condition (2.4) of \citet{Mikosch:Zhao:2014},  among others. Such conditions 
may lead to more restrictions on the underlying dependence than the strong mixing condition itself.

Recently, \citet{Zhang:2021} introduced the notion of tail adversarial stability, which provides an alternative framework for developing asymptotic theories of  analysis of  tail dependent time series. Compared with the traditional strong mixing framework that involves a supremum distance between two sigma algebras, the tail adversarial stability framework of \citet{Zhang:2021} relies on the tail adversarial effect of coupled innovations expressed as a conditional probability, which is much more tractable. It has been shown in \citet{Zhang:2021} and \citet{Zhang:2022} that the tail adversarial stability measure can be easily calculated for the max-linear (or say moving-maximum) processes (see, e.g., \cite{hsing1986extreme} and \citet{Hall:Peng:Yao:2002}), and can lead to cleaner and weaker conditions than the traditional strong mixing framework. Besides the max-linear process and its variants covered in the recent review of \citet{ZhangZ:2021}, the additive autoregressive and moving-average (ARMA) model with heavy-tailed innovations has also been a popular choice for modeling time series with tail dependence. The practical meaning of tail adversarial stability for such additively structured processes, however, has not been addressed so far. This article aims to fill this gap by verifying the tail adversarial stability condition for regularly varying (additive) linear processes, or say moving-average processes.  Note that the causal ARMA processes are   covered due to their moving-average representations. In particular, we develop probabilistic bounds that are uniform across all lags at which the underlying process is decoupled (Lemma \ref{Lem:control} and Corollary \ref{Cor:key}). Based on these  bounds, we then obtain a bound of the tail adversarial stability measure in terms of the coefficients of the linear process (Theorem \ref{Thm:linear}). We hence are able to verify the tail adversarial stability condition under relatively mild assumptions (Corollary \ref{Cor:TAS Linear}). Parallel results for their stochastic volatility extensions are presented in Theorem \ref{Thm:Vol} and Corollary \ref{Cor:TAS cond stoch vol}. We also revisit the the class of max-linear processes, which extends the additive linear process to its maximal counterpart.

The remaining of the article is organized as follows. Section \ref{sec:TAS} reviews the tail adversarial stability measure. Section \ref{sec:linearprocesses} provides an explicit calculation of the tail adversarial stability measure for regularly varying linear processes. Section \ref{sec:stochasticvolatilitymodel} concerns an extension to stochastic volatility type models that are driven by regularly varying linear processes. Section \ref{sec:mmrevisit} revisits the max-linear process when the innovations are general regularly varying random variables. Section \ref{sec:mono} addresses the invariance   of tail adversarial stability under monotone transforms. Section \ref{application} discusses some  implications of our results in statistical context. Section \ref{sec:conclusion} concludes the paper.

\section{Tail Adversarial Stability}\label{sec:TAS}
The notion of tail adversarial stability is formulated for a stationary process $X=(X_i)$ of the following form
\begin{equation}\label{eqn:gen proc}
X_i = g(e_i,e_{i-1},\ldots,e_1,e_0,e_{-1},\ldots),   \quad i =0,1,2,\ldots,
\end{equation}
  where  $g$ is a measurable function,  $(e_j)_{j\in \bb{Z}}$ is a sequence of independent and identically distributed (i.i.d.) random variables, or more generally, random elements such as random vectors.  The framework \eqref{eqn:gen proc} covers a wide range of  common linear and nonlinear time series models.
Let $e_0^*$ be an innovation that has the same distribution as $e_0$ but independent of $(e_j)_{j \in \mathbb Z}$. Then
\begin{equation*}
X_i^* = g(e_i,e_{i-1},\ldots,e_1,e_0^*,e_{-1},\ldots),   \quad i \ge 0
\end{equation*}
represents the coupled observation at time $i$  whose innovation at time zero is replaced by an i.i.d.\ copy. The difference between $X_i$ and the coupled version $X_i^*$, measured through, e.g., the $L^p$ norm $\|X_i-X_i^*\|_p$, leads to the  \emph{physical or functional dependence measure}   formulated by \citet{Wu:2005}. It has since been applied extensively for establishing various  asymptotic results for analysis of dependent data; see for example \citet{Wu:2005}, \citet{Wu:2007:SIP}, \citet{Liu:Wu:2010}, \citet{Zhou:Wu:2010}, \citet{Zhang:Wu:2011}, \citet{Zhang:2013}, and \citet{Zhang:2015}, among others.

Let $\mathcal{U}_X=\inf\{x\in \mathbb{R}:\ \pr(X_0\le x)=1\}\in (-\infty,+\infty]$ be the upper end point of the marginal distribution of $X_0$.
 Now the \emph{tail adversarial stability} (TAS) measure, first introduced in \citet{Zhang:2021} (although in a triangular array setting), is given by
\begin{equation}\label{eq:TAS measure}
\theta_y(i)= \theta_y^{(X)}(i) = \sup_{z \geq y} \pr(X_i^* \leq z \mid X_i > z),\quad y<\mathcal{U}_X,
\end{equation}
which quantifies the adversarial effect of a perturbation of the  innovation $e_0$ has on whether the observed data at time $i$ is an upper-tail observation. Note that if $X_i$ does not depend on $e_0$, then $X_i^* = X_i$ and $\theta_y(i) = 0$, meaning that $e_0$ will not have any tail adversarial effect on $X_i$. Let
\begin{equation}\label{eq:TAS stable measure}
\Theta_{y,q}= \Theta_{y,q}^{(X)} = \sum_{i=0}^\infty \{\theta_y(i)\}^{1/q},\quad q>0,
\end{equation}
which measures the cumulative tail adversarial effect of $e_0$ on all future observations. Then the process $(X_i)$ is said to be  \emph{tail adversarial $q$-stable} or $(X_i) \in \mathrm{TAS}_q$, if
\begin{equation}\label{eq:TAS}
\lim_{y \uparrow \mathcal{U}_X} \Theta_{y,q} < \infty.
\end{equation}

\citet{Zhang:2021} obtained the consistency and central limit theorem for high quantile regression estimators when the underlying process is $\mathrm{TAS}_2$. \citet{Zhang:2022} established the consistency and central limit theorem for sample tail autocorrelations when $(X_i) \in \mathrm{TAS}_q$ for some $q > 4$. We anticipate that more statistical theories for analysis of tail-dependent time series can be developed under the TAS framework.

In this article,  we shall mostly consider the case where  $\pr(X_0>z)>0$ for all $z>0$, that is, the distribution of $X_i$ is unbounded on the positive side so that $\mathcal{U}_X=\infty$, for which it suffices to consider $y>0$ in \eqref{eq:TAS measure}.


\section{Regularly Varying Linear Processes}\label{sec:linearprocesses}
\subsection{Basic Setup}\label{sec:basic setup}
Let $(a_j)_{j\ge 0}$ be a sequence of real coefficients, we consider the additive linear process
\begin{equation}\label{eq:linear proc}
X_i = \sum_{j=0}^\infty a_j \epsilon_{i-j},   \quad i \in\bb{Z},
\end{equation}
where $(\epsilon_j)_{j\in \bb{Z}}$ is a sequence of i.i.d.\ innovation random variables.  It is also known as a  (possibly infinite-order) moving-average process, which includes finite-order ARMA processes as special cases due to their moving-average representations. The process (\ref{eq:linear proc}) is covered by the form in (\ref{eqn:gen proc}) when we set the function $g$ to be linear which is measurable (see, e.g., \citet[Example 2.1.9]{samorodnitsky:2016:stochastic}),  and the random elements $e_i = \epsilon_i$, $i \in \mathbb Z$.

We assume that $|\epsilon_0|$ is regularly varying with index $-\nu$, that is,
\begin{equation}\label{eq:abs RV}
\pr(|\epsilon_0| > x) = x^{-\nu} \ell(x),
\end{equation}
where the function $\ell(\cdot)$   is slowly varying at $\infty$, namely, $\ell(\cdot)$ is a positive function such that $\lim_{z\rightarrow\infty} \ell(\lambda z)/\ell(z)=1$ for any $\lambda>0$.
 See \citet{bingham:goldie:teugels:1989:regular} for more   details about regularly and slowly varying functions. In addition, we assume that the  distribution of $\epsilon_0$ satisfies a   tail balance condition:
\begin{equation}\label{eq:tail balance}
\lim_{x \to \infty} {\pr(\epsilon_0 > x) \over \pr(|\epsilon_0| > x)} = p
\end{equation}
for some $p \in [0,1]$. This is equivalent to assuming that $\lim_{x\rightarrow\infty}\pr(\epsilon_0<-x)/\pr(|\epsilon_0|>x)=1-p$. A distribution of $\epsilon_0$ satisfying both \eqref{eq:abs RV} and \eqref{eq:tail balance} is often known as being \emph{balanced regularly varying}, where $p$ is  the \emph{tail balance parameter}.
 Note that if $\epsilon_0$ comes from a positive distribution such as the Fr\'{e}chet distribution, then for any $x > 0$, $\pr(\epsilon_0 > x) = \pr(|\epsilon_0| > x)$ and hence it has  $p = 1$.

 On the other hand, for the coefficient sequence $(a_j)$ in (\ref{eq:linear proc}), we assume that for some $\varepsilon>0$,
\begin{equation}\label{eq:summability}
\sum_{j=0}^\infty |a_j|^{\frac{\nu}{\nu +2}-\varepsilon  }<\infty.
\end{equation}
Since the exponent $\frac{\nu}{\nu +2}-\varepsilon<\min(1,\nu)$,   the random   series in (\ref{eq:linear proc}) converges almost surely and hence the resulting linear process is well defined; see for example \citet[Corollary 4.2.12]{samorodnitsky:2016:stochastic}. It is worth noting that as $\nu\rightarrow\infty$ and $\varepsilon\rightarrow 0$, the condition \eqref{eq:summability} approaches  $\sum_{j=0}^\infty |a_j|<\infty$, a well-known condition of short memory for linear processes.  The    condition   \eqref{eq:summability} is imposed for establishing  Lemma \ref{Lem:control}, an important uniform estimate for verifying the TAS condition for general regularly varying linear processes. On the other hand, we shall   mention    in Remark \ref{Rem:uniform bound stable distr}  below that when $(\epsilon_j)$ are $\nu$-stable innovations (here $\nu\in (0,2)$), the restriction \eqref{eq:summability} can be relaxed.
We also  assume without loss of generality  that $a_j \neq 0$ for infinitely many $j \geq  1$. The case when $a_j \neq 0$ for  finitely many $j\ge 1$ degenerates to an $m$-dependent process, for which $\Theta_{y,q}  \leq m+1$ and the process  trivially belongs to $\mathrm{TAS}_q$ for any $q > 0$.

Given the assumptions made above, it is known (e.g., \citet[Corollary 4.2.12]{samorodnitsky:2016:stochastic}) that  each $ X_i $ is also balanced regularly varying with index $-\nu$ and
\begin{equation}\label{eq:tail linear}
\lim_{x\rightarrow\infty}\frac{\pr(X_0>x)}{\pr(|\epsilon_0|>x)} =  \sum_{j\ge 0} \left\{ p (a_j)_+^{\nu} + (1-p) (a_j)_-^{\nu} \right\}
\end{equation}
as $x\rightarrow\infty$, where $(\cdot )_+$ and $(\cdot)_-$ stand for positive and negative part respectively.

 We   make an additional mild  assumption: the density $f_\epsilon$ of $\epsilon_0$ exists and satisfies for  any    $\delta\in (0,\nu+1)$, there exists a constant $c_0>0$, such that
\begin{equation}\label{eq:f_eps assumption}
 f_\epsilon(x)  \le c_0  \min(1,|x|^{-\nu -1+\delta}).
\end{equation}

\begin{Rem}
In view of Karamata's Theorem (\citet[Proposition 1.5.8]{bingham:goldie:teugels:1989:regular}) and Potter's bound (\citet[Theorem 1.5.6]{bingham:goldie:teugels:1989:regular}), all the assumptions made so far on $\epsilon_0$ are satisfied if $f_\epsilon$ is bounded, and either both $f_\epsilon(x)$ and $f_\epsilon(-x)$ are regularly varying with index $-\nu-1$ as $x\rightarrow\infty$  with $\lim_{x\rightarrow\infty}f_\epsilon(x)/f_\epsilon(-x)$ existent and positive, or  $f_\epsilon(x)$ is regularly varying with index $-\nu-1$ on one side and is of smaller order on the other side (which corresponds to $p=0$ or $1$ in \eqref{eq:tail balance}). These conditions cover  a broad family of  power-law distributions  such as Pareto, Fr\'echet, Student-t, F-distributions (with the numerator degree of freedom $\ge 1$) and  non-Gaussian  stable distributions (including Cauchy).
\end{Rem}

\subsection{Preparations}

Throughout the article we use   $c$ to denote a generic positive constant whose value   may change from one expression to another.  In this section, we   collect some important auxiliary results   we need for the rest of the article.

The following lemma collects some variants of the Potter's bound   useful for handling regularly varying tails.  Recall a random variable $Z\ge 0$ is said to be regularly varying with index $-\nu$, $\nu>0$, if $\lim_{z\rightarrow\infty} \pr(Z>\lambda z)/\pr(Z>z)=\lambda^{-\nu}$ for any $\lambda>0$.
\begin{Lem}\label{Lem:Potter}
Suppose random variable $Z\ge 0$ is regularly varying with index $-\nu$, $\nu>0$.
Given any fixed $\varepsilon>0$ (and in addition $\varepsilon<\nu$ for \eqref{eq:Potter upper} below), $z_0>0$ and $x_0>0$, there exists a constant  $c>0$, such that
\begin{equation}\label{eq:Potter upper}
\pr(Z>z)\le c z^{-\nu+\varepsilon}, \quad z>0,
\end{equation}
\begin{equation}\label{eq:Potter lower}
\pr(Z>z)\ge c z^{-\nu-\varepsilon}, \quad z>z_0,
\end{equation}
and
\begin{equation}\label{eq:Potter ratio}
\frac{\pr(xZ>z)}{\pr(Z>z)}\le c x^{\nu-\varepsilon}, \quad z>z_0,\quad  x\in[0, x_0].
\end{equation}
\end{Lem}
\begin{proof}
The lemma follows readily from  \citet[Propositions 1.4.1 and 1.4.2]{kulik2020heavy}.
\end{proof}

We also need the following fact on the (truncated) moments of regularly varying random variables.   Below and throughout, we write $E[Z;A]=\E[Z1_A]$ for random variable $Z$, event $A$ and indicator $1_A$.
\begin{Lem}\label{Lem:moment RV}
Suppose random variable $Z\ge 0$ is regularly varying with index $-\nu$, $\nu>0$.  If $\beta\in (0,\nu)$, then $\E[Z^\beta]<\infty$; if $\beta>\nu$, then
\begin{equation*}
\lim_{z\rightarrow\infty}\frac{\E[Z^\beta; Z\le  z]}{z^{\beta}\pr(Z>z)}=\frac{\nu}{\beta-\nu}  >0;
\end{equation*}
In addition,
\begin{equation*}
  \lim_{z\rightarrow\infty}\frac{\E[Z^\nu; Z\le  z]}{z^{\nu+\varepsilon}\pr(Z>z)}=0
\end{equation*}
for any $\varepsilon>0$.
\end{Lem}
\begin{proof}
The first two claims directly follows from \citet[Proposition 1.4.6]{kulik2020heavy}. By  \citet[Proposition 1.4.6]{kulik2020heavy} again, $\E[Z^\nu; Z\le z]$ is slowly varying as $z\rightarrow\infty$, and so is $z^{\nu}\pr(Z>z)$. Hence the last conclusion follows from   \citet[Proposition 1.3.6]{bingham:goldie:teugels:1989:regular}.
\end{proof}

Following Section \ref{sec:TAS}, let $\epsilon_0^*$ be a random variable with the same distribution as $\epsilon_0$ but independent of $(\epsilon_j)_{j \in \mathbb Z}$. Then the coupled version of $X_i$ is
\begin{equation}\label{eq:X_i* linear}
X_i^* = X_i-a_i\epsilon_0+a_i\epsilon_0^* = a_0\epsilon_i + \cdots + a_{i-1}\epsilon_1 + a_i\epsilon_0^* + a_{i+1}\epsilon_{-1} + \cdots
\end{equation}
Introduce
\begin{equation}\label{eq:Y_i}
Y_i=  X_i-a_i\epsilon_0 =\sum_{j\ge 0,\, j\neq i} a_j\epsilon_{i-j}.
\end{equation}
and hence
\[
X_i= Y_i + a_i \epsilon_0,\quad X_i^*= Y_i + a_i \epsilon_0^*.
\]
Below we  develop a uniform estimate  for the densities of $\{Y_i\}$ which will be the key for establishing the main results.  The condition \eqref{eq:summability} plays an important role in an infinite-order induction argument.

 \begin{Lem}\label{Lem:control}
Fix any $\delta\in (0,\nu+1)$. Suppose that $a_j\neq 0$ for infinitely many $j\ge 0$. Under the assumptions \eqref{eq:abs RV}, \eqref{eq:tail balance}, \eqref{eq:summability} and \eqref{eq:f_eps assumption}.
The density $f_i$ of each $Y_i$, $i\ge 0$, exists, and we have for all $x\in \bb{R}$ and $i\ge 0$,
\begin{equation}\label{eq:f_i bound second}
f_i(x)   \le
  c \min(1, |x|^{-\nu-1+\delta})
\end{equation}
 for some positive constant  $c>0$ that does not depend on $i$ or $x$.
\end{Lem}
\begin{Rem}
Observe  that \eqref{eq:f_i bound second} is equivalent to imposing both the uniform boundedness $\sup_i \sup_x f_i(x)<\infty$   and the uniform power decay $\sup_i f_i(x)= O(|x|^{-\nu-1+\delta})$ as $|x|\rightarrow\infty$ for any $\delta\in (0,\nu+1)$.
\end{Rem}
\begin{proof}[Proof of Lemma \ref{Lem:control}]
We shall assume    that $a_j\neq 0$ for every $j\ge 0$. Otherwise, if $a_i=0$, $i\ge 0$, then $Y_i=X_i$ with $f_i$ being the marginal density of $X_0$.  The proof below with a slight modification readily covers  this case. We also assume that $|a_j|< 1.$ Otherwise, apply a proper scaling.

We first prove the existence   and the uniform boundedness of $f_i$.
Suppose first $i\ge 1$. Use $P_{i,1}$ to denote the distribution of $Z_i:=\sum_{j\ge 1,j\neq i}a_{j}\epsilon_{i-j}$. In view of Fubini's theorem, the  density   of $Y_i=a_0\epsilon_i+ Z_i$   exists and can be identified with   the convolution
\[
f_i(x): =\int_{\bb{R}} |a_0|^{-1} f_\epsilon( a_0^{-1}(x-y))P_{i,1}(dy).
\]
Hence  by \eqref{eq:f_eps assumption}, writing $\|g\|_\infty=\sup_{x\in \bb{R}}|g(x)|$ for a function $g$,  we have for $i\ge 1$,
\[
\|f_i\|_\infty\le\int |a_0|^{-1}  \|f_\epsilon\|_\infty P_{i,1}(dy)= |a_0|^{-1} \|f_\epsilon\|_\infty\le |a_0|^{-1} c_0<\infty.
\]
 The existence and boundedness for $f_0$ can be obtained similarly by replacing the role of $a_0\epsilon_i$ by $a_1 \epsilon_{i-1}$ in the argument above. Hence the uniform boundedness required in \eqref{eq:f_i bound second} follows.

Now we turn to the uniform power decay   in \eqref{eq:f_i bound second}. Recall $\delta$ in \eqref{eq:f_eps assumption} can be specified arbitrarily small.   Since  $ \min(1, |x|^{-\nu-1+\delta})$ is non-decreasing with respect to $\delta\in (0,\nu+1)$,  it suffices to prove \eqref{eq:f_i bound second} for any sufficiently (to be specified later) small  $\delta\in (0, \nu)$.  We only prove the  uniform power decay on the positive side $x>0$, and the case $x<0$ follows similarly.
 Set
\[
 h_\kappa(x)=c  \min(1,|x|^{-\kappa}),\quad \kappa:=\nu+1-\delta>1,
\]
where $c>0$ is a constant such that  (see \eqref{eq:f_eps assumption})
\begin{equation}\label{eq:f_eps bound h}
f_\epsilon(x)\le h_\kappa(x).
\end{equation}
Now for a function $g\ge 0$, we define
\[
\cl{M}^\kappa  g =\sup_{t>0} t^\kappa g(t).
\]
To prove the uniform power decay in \eqref{eq:f_i bound second} on the positive side, it suffices to show
\begin{align}\label{eq:f_i uniform power decay goal}
\sup_{i\ge 0}  \cl{M}^\kappa f_i <\infty.
\end{align}

 Suppose first $i\ge 2$.

Below we apply an infinite-order induction argument    similar to the proof of \citet[Lemma 6.6.3]{barbe2009asymptotic}.
For a fixed constant $\rho>0$ to be specified later, let \[d_j =|a_j|^{\rho}\in \left(0,1\right),\quad j\ge 1. \]
Let  $g_{i,n}$ be the density of $\sum_{0\le j\le  n,j\neq i}a_{j}\epsilon_{i-j}$, $n\ge 0$, the truncated approximation of $Y_i$. Note that  $g_{i,0}(\cdot)=|a_0|^{-1}f_\epsilon(a_0^{-1}\cdot)$ since we have supposed $i\ge 2$.

If $n\neq i$ and $n\ge 1$,  we   decompose the convolution $g_{i,n}= g_{i,{n-1}}* \left(|a_n|^{-1}f_\epsilon(a_{n}^{-1} \cdot)\right)$  as
\begin{align}
&g_{i,n}(x)  =\int_{-\infty}^{d_n x} g_{i,{n-1}}(x-y)  |a_{n}|^{-1} f_{\epsilon} (a_{n}^{-1}y)  dy +\int_{-\infty}^{(1-d_n)x} g_{i,n-1}(y)  |a_{n}|^{-1} f_{\epsilon} (a_{n}^{-1}(x-y))  dy. \label{eq:g bound}
\end{align}
Note that
\begin{align*}
\sup_{y<d_n x } g_{i,n-1}(x-y)\le &  \left(\sup_{t>(1-d_n)x} t^{-\kappa}\right) \left( \sup_{t>(1-d_n)x} t^{\kappa}g_{i,n-1}(t)\right)    \le    (1-d_n)^{-\kappa}x^{-\kappa}  \cl{M}^\kappa g_{i,n-1} .
\end{align*}
Therefore,  since   $\int_{\bb{R}} |a_{n}|^{-1} f_{\epsilon} (a_{n}^{-1}y)dy=1$, we have for all $x>0$ that
\begin{align*}
\int_{-\infty}^{d_n x} g_{i,{n-1}}(x-y)  |a_{n}|^{-1} f_{\epsilon} (a_{n}^{-1}y)  dy\le   (1-d_n)^{-\kappa} x^{-\kappa}  \cl{M}^\kappa g_{i,n-1}. \end{align*}
On the other hand, by the bound \eqref{eq:f_eps bound h} and the symmetry of $h_\kappa$, we have   $\sup_{y<(1-d_n)x} f_\epsilon(a_n^{-1}(x-y))\le \sup_{y<(1-d_n)x}  h_{\kappa}(|a_{n}|^{-1}(x-y)) \le c_0 (|a_n|^{-1} d_n x)^{-\kappa} $. Hence using $\int_{\bb{R}}g_{i,n-1}(y)dy=1$,  we have for all $x>0$ that
\begin{align*}
\int_{-\infty}^{(1-d_n)x} g_{i,n-1}(y)  |a_{n}|^{-1} f_{\epsilon} (a_{n}^{-1}(x-y))  dy\le c    |a_n|^{\kappa-1} d_n^{-\kappa} x^{-\kappa}.
\end{align*}
Applying the two   displayed bounds  above to \eqref{eq:g bound},  we conclude that
\begin{equation}\label{eq:g ind bound}
 \cl{M}^\kappa g_{i,n} \le   (1-d_n)^{-\kappa}    \cl{M}^\kappa g_{i,n-1}    +  c |a_n|^{\kappa-1} d_n^{-\kappa}.
\end{equation}

If $n=i\ge 2$, then  $g_{i,n}=g_{i,n-1}$, and  the bound above trivially follows from monotonicity.

Define
for $1\le j+1\le n$ that
\[B_{j,n}=\prod_{j +1 \le \ell \le n} (1-d_\ell)^{-\kappa}.
\]
Set also $B_{n,n}=1$.
Now by  an induction   based on the  recursive bound \eqref{eq:g ind bound}, it can be verified that for all $n\ge 1$,
\begin{align*}
 \cl{M}^\kappa g_{i,n} \le B_{0,n}  \cl{M}^\kappa g_{i,0}  +   c   \left(\sum_{j=1}^n  B_{j,n} |a_j|^{\kappa-1} d_j^{-\kappa} \right).
\end{align*}
Note that   $B_{j,n}$ increases as $j$ decreases or as $n$ increases. In view of the monotonicity,  we have
\begin{align}\label{eq:max fun bound}
 \cl{M}^\kappa g_{i,n} \le B   \cl{M}^\kappa g_{i,0} +   c  B \left(\sum_{j=1}^\infty  |a_j|^{\kappa-1} d_j^{-\kappa} \right),
\end{align}
where  $B=\lim_n B_{0,n}$. Now set $ \rho=\kappa(1-\rho)-1$ which implies \[\rho=\frac{\kappa-1}{\kappa+1}=\frac{\nu-\delta}{\nu+2-\delta}.\]  By assumption  \eqref{eq:summability}  with  $\delta>0$ chosen sufficiently small so that $\rho\ge \frac{\nu}{\nu+2}-\varepsilon$, we have
\[
\sum_{j=0}^\infty  d_j =\sum_{j=0}^\infty |a_j|^{\rho}<\infty, \quad \text{and }\ \sum_{j=1}^\infty  |a_j|^{\kappa-1} d_j^{-\kappa}=\sum_{j=1}^\infty  |a_j|^{\kappa(1-\rho)-1}<\infty.
\]
  Note that the summability $\sum_{j=0}^\infty  d_j <\infty$ implies  $ B\in (0,\infty)$. Recall $g_{i,0}(\cdot)=|a_0|^{-1}f_\epsilon(a_0^{-1}\cdot)$, and hence by \eqref{eq:f_eps bound h} and \eqref{eq:max fun bound} we have
\begin{equation}\label{eq:g uniform bound}
\sup_{n\ge 1,i\ge 2}  \cl{M}^\kappa g_{i,n}  <\infty.
\end{equation}
Let $P_{i,2,n}$ denote the distribution of  of $Z_{i,2,n}:=\sum_{2\le j\le n,\, j\neq i}a_{j}\epsilon_{i-j}$, $n\ge 2$. The a.s.\ convergence of $Z_{i,2,n}$ to $Z_{i,2}:=\sum_{j\ge 2,\, j\neq i}a_{j}\epsilon_{i-j}$  implies the weak convergence of $P_{i,2,n}\Rightarrow P_{i,2}$ as $n\rightarrow\infty$, where $P_{i,2}$ is the distribution of $Z_{i,2}$.  On the other hand, the function $g_{i,1}$  (recall $i\ge 2$), as a convolution between two bounded integrable functions $|a_0|^{-1}f_\epsilon(a_0^{-1}\cdot)$ and $|a_1|^{-1}f_\epsilon(a_1^{-1}\cdot)$,  is   bounded and  continuous (e.g., \citet[Corollary 3.9.6]{bogachev2007measure}). Hence  for any $x\in \bb{R}$ and $n\ge 1$, we have for all $x>0$ that
\begin{equation}\label{eq:limig g_i,n}
g_{i,n}(x)=\int_{\bb{R}} g_{i,1}(x-y) P_{i,2,n}(dy)\rightarrow \int_{\bb{R}} g_{i,1}(x-y) P_{i,2}(dy)=f_i(x)
\end{equation}
as $n\rightarrow\infty$. So combining \eqref{eq:limig g_i,n} with \eqref{eq:g uniform bound}, we conclude that
\begin{align}\label{eq:f_i bound goal}
\sup_{i\ge 2}  \cl{M}^\kappa f_i <\infty.
\end{align}

By a similar argument which uses some other indices to replace the roles of $i=0,1$ above, we can also show that  $ \cl{M}^\kappa f_i <\infty$ for $i=0,1$. This combined with \eqref{eq:f_i bound goal}  concludes    \eqref{eq:f_i uniform power decay goal}.
\end{proof}

\begin{Rem}
It is possible to further improve the tail decay in the bound \eqref{eq:f_i bound second}. Under  additional assumptions including certain  smooth (regular) variation (cf., \cite[Section 1.8]{bingham:goldie:teugels:1989:regular}) conditions    on the distribution of $\epsilon_0$, the  remarkable work of
\cite{barbe2009asymptotic} developed uniform asymptotic expansions for the marginal distribution of regularly varying linear series.  In paticular, their Theorem 2.5.1 implies a sharp uniform bound: there exists $y_0>0$, such that for all $y>y_0$,
\begin{equation*}
\sup_{  i\ge 0}f_i(y)   \le
  c   |y|^{-\nu-1}  \ell(y),\quad i\ge 0,
\end{equation*}
with the slowly varying $\ell(y)$ as in \eqref{eq:abs RV}.       On the other hand, the slightly weaker tail bound  in \eqref{eq:f_i bound second} is verified  under less stringent assumptions compared to \cite{barbe2009asymptotic}, which suffices for our purposes.
\end{Rem}

\begin{Rem}\label{Rem:uniform bound stable distr}
The   summability assumption \eqref{eq:summability} is imposed for establishing the uniform  power decay estimate  in \eqref{eq:f_i bound second}. On the other hand,  the restriction  \eqref{eq:summability} is likely  only an artifact of the current proof. Indeed,
consider the case where the innovations $(\epsilon_i)$ are standard  symmetric $\nu$-stable (S$\nu$S), $\nu\in (0,2)$, specified by the characteristic function   $\E[e^{i\theta\epsilon_0}]=e^{-|\theta|^{ \nu}}$. See, e.g., \citet{nolan2020univariate} for more details.  This  special case satisfies the balanced regularly variation assumptions \eqref{eq:abs RV} and  \eqref{eq:tail balance}.   In this case, for the linear process \eqref{eq:linear proc} to be well-defined, it suffices to assume $\sum_{i=0}^\infty |a_i|^{ \nu\wedge 1}<\infty$ (e.g., \citet{kokoszka1995fractional}). In addition, it follows from the sum-stability property of S$\nu$S distributions that $Y_i\EqD\left(\sum_{j\ge 0,j\neq i}|a_i|^{\nu}\right)^{1/\nu}\epsilon_0$, where $\inf_{i\ge 0}\left(\sum_{j\ge 0,j\neq i}|a_i|^{\nu}\right)^{1/\nu}>0$. This implies   a uniform bound      \[f_i(y)\le c\min(1,|y|^{-\nu-1})\] for all $i\ge 0$ (cf., e.g., Sections 1.4 and 1.5 of \citet{nolan2020univariate}). The  discussion above is generalizable to   non-symmetric stable  distributions which for simplicity is omitted.
\end{Rem}

\begin{Cor}\label{Cor:key}
Fix any $\delta\in (0,\nu+1)$.
Under the same   assumptions as Lemma \ref{Lem:control},
  for all $-  |z|/2\le  v\le u\le  |z|/2$, $z\in \mathbb{R}$ and $i\ge 0$, we have
\[
 \pr( Y_i\in [z-u,z-v])  \le    c (u-v)  \min\left(1, |z|^{-\nu-1+\delta}\right).
\]
for some  constants $c>0$.
\end{Cor}
\begin{proof}

Write
\begin{align*}
 \pr( Y_i\in [z-u,z-v]) = \int_{z-u}^{z-v} f_i(y)dy.
\end{align*}
By Lemma \ref{Lem:control}, we have the constant bound $f_i(y)\le c$ and the power-law bound $f_i(y)\le c |y|^{-\nu-1+\delta}$ for all $i\ge 0$.
The constant bound     yields
\[
 \pr( Y_i\in [z-u,z-v]) \le c (u-v).
\]
The power-law bound  combined with the restriction on $u$ and $v$    yields
\[
 \pr( Y_i\in [z-u,z-v]) \le c \int_{z-u}^{z-v}   |y|^{-\nu -1+\delta}dy\le c \int_{z-u}^{z-v}   |z/2|^{-\nu -1+\delta}dy\le c |z|^{-\nu-1+\delta}(u-v).
\]
Combining the bounds  concludes the proof.
\end{proof}

\subsection{Verification of TAS Condition}

We shall  provide an explicit bound of the TAS measure \eqref{eq:TAS measure} for the linear process  \eqref{eq:linear proc}  in Theorem \ref{Thm:linear} below. The bound enables an immediate verification of $\mathrm{TAS}_q$ in Corollary \ref{Cor:TAS Linear} below.

 Below is the main result.
\begin{Thm}\label{Thm:linear}Suppose $(X_i)$ is a  linear process as in \eqref{eq:linear proc} with i.i.d.\   innovations $(\epsilon_i)$ satisfying \eqref{eq:abs RV}, \eqref{eq:tail balance} and \eqref{eq:f_eps assumption}.  Assume the coefficients $a_i\neq 0$ for infinitely many $j\ge0$, the summability condition  \eqref{eq:summability} holds, and   the right-hand side of \eqref{eq:tail linear} is nonzero.
Fix an arbitrary $\eta\in (0,\nu)$ when $\nu\le 1$.   There exist   constants $c>0$ and $y_0>0$,   such that for all    $y>y_0$ and $i\ge 0$,   the TAS measure $\theta_y(i)$  in \eqref{eq:TAS measure} satisfies
\[
\theta_y(i)   \le  \begin{cases}
c    |a_i|,     & \nu>1,\\
 c     \left| a_i  \right|^{ \eta},    & \nu\le 1,
\end{cases}
\]
\end{Thm}
In fact, Theorem \ref{Thm:linear} follows   from Theorem \ref{Thm:Vol} below in the special case where each $S_i=1$ in \eqref{eq:R_i}.  For convenience, we include a separate and more transparent proof for this special case below.
\begin{proof}
Assume without loss of generality that every $a_j\neq 0$ and let $z>0$. Below $y_0>0$ is a constant which does not depend on $i\ge 0$, whose value may be increased if necessary each time  when mentioned.
Write
\begin{align*}
& \pr( X_i^* \le z  \ |\    X_i > z) = \frac{\pr\left(z- a_i  \epsilon_0 <Y_i\le z-a_i \epsilon_0^*\right)}{ \pr( X_i >z) }\\
\le & \frac{\pr\left( z- a_i  \epsilon_0 <Y_i\le z-a_i \epsilon_0^* ,\  - z /2 \le a_i \epsilon_0^*< a_i \epsilon_0\le  z/2  \right)}{ \pr( X_0 >z) }+\frac{  P\left( |a_i\epsilon_0|> z/2   \right)}{ \pr( X_0 >z) }\\=&: A_i(z) +B_i(z).
\end{align*}

\noindent $\bullet $ Suppose $\nu>1$.

Recall that $\pr(X_0>z)$ is regularly varying with index $-\nu$ in view of \eqref{eq:tail linear}. Hence by Potter's bound \eqref{eq:Potter lower}, for any chosen  $\delta\in (0,1)$, there exists $y_0>0$ and constant $c>0$ such that
\[
\pr(X_0>z)\ge c z^{-\nu-1+\delta}
\]
for all $z>y_0$. Then by independence and Corollary \ref{Cor:key} with the same $\delta>0$, we have   for all $z>y_0$,
\begin{align*}
 A_i(z) = & \int_{- z / 2 \le a_i v< a_i u \le z /2}  \frac{\pr\left(z- a_i  u \le Y_i\le z- a_i  v \right)}{ \pr( X_0 >z)}  f_\epsilon(u)f_\epsilon(v)  du dv  \\
\le &   c |a_i|    \int_{ |a_iu|, |a_iv| \le z/2  }  |u-v|  f_\epsilon(u)f_\epsilon(v)  du dv.
\end{align*}
We can bound the integral above as
\begin{equation}\label{eq:int exp}
\int_{|a_iu|, |a_iv| \le z/2}  |u-v|  f_\epsilon(u)f_\epsilon(v)  du dv \le   \E  |\epsilon_0-\epsilon^*_0| \le 2\E|\epsilon_0|.
\end{equation}
Hence  if $\nu >1$ under which $ \E|\epsilon_0|<\infty$, we have
\[
 A_i(z)\le c   | a_i  |
\]
for $z>y_0$. On the other hand, it follows from \eqref{eq:tail linear}, the restriction $\nu>1$ and   Potter's bound \eqref{eq:Potter ratio} that  for all $z>y_0$,
\[
B_i(z)\le    c |a_i|.
\]

\noindent $\bullet$ Suppose $\nu\in (0,1]$.

By the Potter's bound \eqref{eq:Potter lower}, for any $\delta$ chosen such that $0<\delta<\eta<\nu$, we have
\[
\pr(X_0>z)\ge c z^{-\nu -\eta+\delta}
\]
for for all $z>y_0$.
Then similarly as above, applying Corollary \ref{Cor:key} with the same $\delta$, we have for $z>y_0$ that
\begin{align*}
 A_i(z)
\le     c |a_i|  z^{\eta-1}   \int_{ |a_iu|, |a_iv| \le z/2  }  |u-v|  f_\epsilon(u)f_\epsilon(v)  du dv.
\end{align*}
 For all $z>y_0$,   the integral above is bounded by
\begin{equation}\label{eq:trun eps moment}
2\E [|\epsilon_0|; |\epsilon_0|\le z/(2|a_i|) ]\le c   \left(\frac{z }{|a_i|}\right)^{1-\eta}.
\end{equation}
The last   inequality follows from  $\sup_{j}|a_j|<\infty$,   Lemma \ref{Lem:moment RV} and Potter's bound  \eqref{eq:Potter upper}.  Then for $z> y_0$,
\[
A_i(z)\le   c |a_i|^{\eta}.
\]
  On the other hand, it follows from \eqref{eq:tail linear} and  Potter's bound \eqref{eq:Potter ratio} that  for all $z>y_0$,
\[
B_i(z)\le c\frac{  \pr\left( |a_i\epsilon_0|> z/2   \right)}{ \pr( |\epsilon_0|>z) } \le  c |a_i|^\eta.
\]
The conclusion follows.

\end{proof}

Under the conditions of the theorem above,
we have for any $\eta\in(0,\nu)$ arbitrarily close to $\nu$,
\[
\Theta_{y,q}=\sum_{i=0}^\infty\theta_y(i)^{1/q} \le \begin{cases}
c\sum_{i=0 }^\infty |a_i|^{1/q}      & \nu>1,\\
 c  \sum_{i=0}^\infty   |a_i|^{\eta/q}   & \nu\le 1.
\end{cases}
\]Combining this with \eqref{eq:summability}, we arrive at the following sufficient condition for the $\mathrm{TAS}_q$ condition.
\begin{Cor}\label{Cor:TAS Linear}
Suppose the assumptions of Theorem \ref{Thm:linear} holds.
 Then $\mathrm{TAS}_q$ condition \eqref{eq:TAS} holds if for some $\varepsilon>0$,
\[\sum_{i\ge 0} |a_i|^{\vartheta(\nu,q,\varepsilon)}<\infty;
\]
where
\[
\vartheta(\nu,q,\varepsilon)=\begin{cases} \frac{1}{q}, & \text{ when }
\nu>1 \text{ and } \frac{\nu}{\nu +2}>\frac{1}{q};\\
\frac{\nu}{\nu +2}-\varepsilon, &   \text{ when } \nu>1 \text{ and } \frac{\nu}{\nu +2}\le \frac{1}{q};\\
 \frac{\nu}{q} -\varepsilon, & \text{ when } \nu\in (0,1] \text{ and } \nu<q-2;\\
  \frac{\nu}{\nu+2} -\varepsilon, & \text{ when } \nu\in (0,1] \text{ and } \nu\ge q-2.
\end{cases}
\]

\end{Cor}

\begin{Rem}\label{Rem:conjecture}
As mentioned in Remark \ref{Rem:uniform bound stable distr}, the restriction \eqref{eq:summability} can be relaxed when the innovations $(\epsilon_i)$ are S$\nu$S random variables, $\nu\in(0,2)$. In this case,  it follows a similar line of argument as the proof of  Theorem \ref{Thm:linear} and the   properties of stable distributions   that the $\mathrm{TAS}_q$ condition holds  if $\sum_{i\ge 0} |a_i|^{\vartheta(\nu,q,\varepsilon)}<\infty$ but with  $\vartheta(\nu,q,\varepsilon)$ in Corollary \ref{Cor:TAS Linear}  above    replaced   by
\[
\vartheta(\nu,q,\varepsilon)=\begin{cases} \frac{1}{q}, & \text{ when }
\nu>1;\\
1/q-\varepsilon, &   \text{ when } \nu=1;\\
 \frac{\nu}{q}, & \text{ when } \nu<1;\\
\end{cases}
\] We conjecture that for a large class of regularly varying linear processes,  the uniform estimate in Lemma \ref{Lem:control} holds  under less stringent conditions than  \eqref{eq:summability}, and  that $\mathrm{TAS}_q$ holds under conditions close to the one mentioned above for the S$\nu$S case.
\end{Rem}

\section{A Stochastic Volatility Extension}\label{sec:stochasticvolatilitymodel}
\subsection{Model Setup}\label{sec:stoch vol}
Consider the following model of  stochastic volatility  type. Let $X_i=\sum_{j=0}^\infty a_{j} \epsilon_{i-j}$ be the linear process with innovations $\epsilon_j$ balanced regularly varying  with index $-\nu<0$, which satisfies all the assumptions in     Section \ref{sec:basic setup}. Here we allow the right-hand side of \eqref{eq:tail linear}  to be zero (i.e., left tail of $X_i$ dominates instead). In particular, we have
\begin{equation}\label{eq:tail linear two sides}
\lim_{x\rightarrow\infty}\frac{\pr(X_0>x)}{\pr(|\epsilon_0|>x)} =  A_1,\quad \lim_{x\rightarrow\infty}\frac{\pr(X_0<-x)}{\pr(|\epsilon_0|>x)} = A_2
\end{equation}
where $A_1=\sum_{j=0}^\infty \left( p (a_j)_+^{\nu} + (1-p) (a_j)_-^{\nu} \right)$ and $A_2=\sum_{j=0}^\infty\left( p (a_j)_-^{\nu} +  (1-p) (a_j)_+^{\nu} \right)$,
where either $A_1>0$ or $A_2>0$ since $A_1+A_2=\sum_{j\ge 0} |a_j|^{\nu}>0$.

 Let $(S_i)$ be i.i.d.\ random variables independent of $(\epsilon_j)$.   Then consider the model of stochastic volatility type:
\begin{equation}\label{eq:R_i}
R_i = S_i X_i.
\end{equation}
Note that this follows the causal process form \eqref{eqn:gen proc}   with $ e_i:=(S_i,\epsilon_i)$.

We introduce for notational simplicity   $(S,X)\EqD (S_0,X_0)$ and set $R= SX$. Write also $(\epsilon,S)\EqD(\epsilon_0,S_0)$. Below $Z_+$ and $Z_-$ denote the positive and negative parts of random variable $Z$ respectively. We make the following assumption.
\begin{Ass}\label{Ass:tail compar}
Assume either of the following cases holds.
\begin{enumerate}[(I)]
\item ($S$ has lighter tail than $X$)  For some $\beta>\nu$ that
\begin{equation}\label{eq:breiman moment}
E|S|^{\beta}<\infty,
\end{equation}
and
\begin{equation}\label{eq:non degen}
\pr(A_1 S_+    +  A_2 S_-   >0)>0;
\end{equation}

\item  ($S$ has heavier  tail than $X$)   $S$ is balanced regularly varying  with index $-\beta$ and tail balance parameter $q:=
\lim_{x\rightarrow\infty} \pr(S>x)/\pr(|S|>x) \in [0,1]$, where $\beta\in (0,\nu)$,  and
\begin{equation}\label{eq:non degen II}
\pr(q X_+ +(1-q) X_->0)>0;
\end{equation}

\item  ($S$ has comparable tail as $X$)
 $S$ is balanced regularly varying  with index $-\beta$ and tail balance parameter $q$ as above,
 $\beta=\nu$, and
 \begin{equation}\label{eq:non degen III}
\text{if $A_2=0$, $q>0$; \ if $A_1=0$, $q<1$.}
 \end{equation}

 \end{enumerate}
\end{Ass}

Throughout the paper, we write $a_i \sim b_i$ if $a_i/b_i \to 1$ as $i \to \infty$.

\begin{Rem}
Under Case (I),
by Breiman's Lemma (e.g., \cite[Lemma 1.4.3]{kulik2020heavy}) and \eqref{eq:tail linear} (i.e., \cite[Corollary 4.2.12]{samorodnitsky:2016:stochastic}),  one has  as $z\rightarrow\infty$ that
\begin{equation}\label{eq:R tail I}
\pr(R>z)=\pr(S_+ X_+>z)+\pr(S_- X_->z)\sim \left( A_1 \E S_+^{\nu}  + A_2 \E S_-^{\nu} \right) \pr(|\epsilon_0|>z) ,
\end{equation}
where   $A_1 \E S_+^{\nu}  + A_2 \E S_-^{\nu}>0$ under the assumption \eqref{eq:non degen}  and hence $\pr(R>z)$ is regularly varying with index $-\nu$ as $z\rightarrow\infty$.

Under (II)  when $\beta<\nu$, since $E|X|^{\beta+\gamma}<\infty$ for $\gamma\in (0,\nu-\beta)$ (Lemma \ref{Lem:moment RV}),  by Breiman's Lemma similarly as above,  we have as  $z\rightarrow\infty$ that
\begin{equation}\label{eq:R tail II}
\pr(R>z) =\pr(S_+ X_+>z)+\pr(S_- X_->z)   \sim  \left(q \E X_+^{\beta}+(1-q)\E X_-^{\beta}\right) \pr(|S|>z),
\end{equation}
where $q \E X_+^{\beta}+(1-q)\E X_-^{\beta}>0$ under the assumption \eqref{eq:non degen II} and   hence $\pr(R>z)$ is regularly varying with index $-\beta$ as $z\rightarrow\infty$.

 Under (III) when $\beta=\nu$,  by \cite[COROLLARY of Theorem 3]{embrechts1980closure},   the tail $\pr(|R| >z)$ is a regularly varying with index $-\nu$ as $z\rightarrow\infty$.    However, the same result cannot conclude regular variation of $\pr(R >z)=\pr(S_+ X_+ >z)+\pr(S_-X_->z)$ as $z\rightarrow\infty$ in all the possible cases. For example, when $A_1,A_2>0,q=0$, while $\pr(S_+ >z)=o(\pr(S_- >z))$,  it could happen that $\pr(S_+>z)$ is neither regularly varying  nor of smaller order than $\pr(X_+>z)$ as $z\rightarrow\infty$. In this case, \cite[COROLLARY of Theorem 3]{embrechts1980closure} is not applicable to conclude the regular variation of $\pr(S_+X_+>z)$, although  the regular variation of $\pr(S_-X_->z)$ follows.

  Note that
the condition \eqref{eq:non degen III} excludes the special cases $A_2=q=0$ or $A_1=1-q=0$.  These two cases introduce  some technical difficulty to the current proof. On the other hand, these two special cases possibly allow $R_-=S_+X_-+S_-X_+$ to have  a heavier tail than $R_+=S_+X_+ +S_-X_-$, which is less relevant since the focus is  on the right tail of $R$.

\end{Rem}


\subsection{Verification of TAS Condition}

Let $R_i^*$ be as $R_i$ except that $e_0$ is replaced by an  identically distributed copy $e_0^*=(S_i^*,\epsilon_i^*)$ independent of $(e_i)$.
Define as before
\begin{equation}\label{eq:theta R}
\theta_y(i):=\sup_{z\ge y} \pr( R_i^* \le z  \ |\    R_i > z)
\end{equation}
and then
\[
\Theta_{y,q} :=\sum_{i=0}^\infty\theta_y(i)^{1/q}
\]
It turns out that the same conclusion as Theorem \ref{Thm:linear} holds for the stochastic volatility extension.

\begin{Thm}\label{Thm:Vol}
Suppose $(R_i)$ is of the form \eqref{eq:R_i} with  $(X_i)$ specified as a linear process in \eqref{eq:linear proc} with  the coefficient $a_i\neq 0$ for infinitely many $i\ge 0$,  satisfying \eqref{eq:abs RV}, \eqref{eq:tail balance}, \eqref{eq:summability} and \eqref{eq:f_eps assumption}. Suppose also that Assumption \ref{Ass:tail compar} holds.
Fix an arbitrary $\eta\in (0,\nu)$ when $\nu\le 1$. There exist   constants $c>0$ and $y_0>0$, such that for all  large $y\ge y_0$ and $i\ge 0$, the TAS measure $\theta_y(i)$  in \eqref{eq:theta R} satisfies
\[
\theta_y(i)   \le  \begin{cases}
c    |a_i|      & \nu>1,\\
 c     \left| a_i  \right|^{ \eta}    &  \nu\le 1.
\end{cases}
\]
\end{Thm}
\begin{proof}
Assume without loss of generality $S_i\neq 0$ a.s.\ (otherwise condition on $\{S_i\neq 0\}$) and every $a_j\neq 0$, $j\ge 0$. Recall $(S,X)\EqD(S_i, X_i)$ and $R=SX$, and write  $P_S$   for the distribution of $S$. Suppose $z>0$. Below $y_0>0$ is a constant which does not depend on $i$, whose value may be increased if necessary each time when mentioned.
We have
\begin{align}
&\pr( R_i^* \le z  \ |\    R_i > z)= \frac{\pr\left(z- a_i   S_i\epsilon_0 <S_iY_i\le z -a_i  S_i\epsilon_0^*\right) }{ \pr( R_i >z) } \notag\\\le&   \frac{\pr\left( z- a_i   S_i\epsilon_0 <S_iY_i\le z -a_i  S_i\epsilon_0^* \, ,~  -z /2 \le a_iS_i\epsilon_0^*<a_i S_i\epsilon_0\le z/2 \right)}{ \pr( R >z) } +  \frac{\pr\left( |a_iS\epsilon|> z/2  \right)}{ \pr( R >z) }\notag  \\
=& A_i(z)+B_i(z). \label{eq:two term}
\end{align}

\noindent $\bullet$ Suppose  $\nu>1$.

For some $\delta\in (0,1)$ to be chosen later,    the numerator of   $A_i(z)$ above can be bounded using Corollary \ref{Cor:key} as
\begin{align}
&\int_{|s|\in (0, \infty)}  P_S(ds)\int_{   -z/2 \le a_i sv<a_isu \le z / 2     } \pr(sY_i\in (z-a_isu,z-a_isv] )  f_\epsilon(u)f_\epsilon(v)  du dv \notag\\
\le
&\int_{|s|\in (0, z]} P_S(ds) \int_{    |a_i sv|, |a_isu| \le z / 2     } c |a_i| |u-v|   |z/s|^{-\nu-1+\delta}   f_\epsilon(u)f_\epsilon(v)  du dv \notag\\ +&\int_{|s|\in ( z ,\infty)} P_S(ds)\int_{ |a_i sv|, |a_isu| \le z / 2      } c |a_i| |u-v|    f_\epsilon(u)f_\epsilon(v)  du dv\notag
\\\le & c |a_i|  z^{-\nu-1+\delta}  \E[|S|^{\nu+1-\delta}  ;|S|\le z ]    + c   |a_i| \pr(|S|>z ), \label{eq:two term S}
\end{align}
where in the last inequality above we have applied
\begin{equation}\label{eq:eps abs moment}
\int_{|a_isu|, |a_isv| \le z/2}  |u-v|  f_\epsilon(u)f_\epsilon(v)  du dv \le   \E  |\epsilon_0-\epsilon^*_0| \le 2\E|\epsilon_0|.
\end{equation}

We consider    the Cases  (I)$\sim$(III) in Assumption \ref{Ass:tail compar} separately.

\noindent \emph{Case (I)}.

Suppose  that $\delta\in (0,1)$ is chosen sufficiently close to $1$   so that $\nu+1-\delta\in (0,\beta)$ (recall $\nu<\beta$). Then $\E   |S| ^{\nu+1-\delta} <\infty$ and hence $\pr(|S|>z)\le c z^{-\nu-1+\delta}$ by Markov inequality.   Note that $\pr( R >z)\ge c z^{-\nu-1+\delta}$ when $z>y_0$, which is a consequence of  regular variation of $\pr(R>z)\sim c \pr(| \epsilon|>z)$ of index $-\nu$   as $z\rightarrow\infty$ as described in \eqref{eq:R tail I} and  Potter's bound \eqref{eq:Potter lower}. Combining these facts to   \eqref{eq:two term S} we conclude  that   for $z>y_0$,
\[
A_i(z)   \le  c |a_i|.
 \]
Next, observe that $\pr\left( |S\epsilon|> z  \right)\sim \E |S|^{\nu} \pr(|\epsilon|>z)$ as $z\rightarrow\infty$ by Breiman's Lemma. This implies that for all $z>0$ and $i\ge 0$, we have $\pr\left( |a_iS\epsilon|> z/2  \right)\le c \pr\left( |a_i \epsilon|> z/2  \right)$ for some large enough constant $c>0$.  Combining this with the aforementioned fact $\pr(R>z)\sim c \pr(| \epsilon|>z)$ as $z\rightarrow\infty$,       we have  for $z>y_0$ that
\begin{equation}\label{eq:B_i bound 1<nu<beta}
B_i(z)\le c \frac{\pr\left( |a_i  \epsilon|> z/2  \right)}{ \pr( |\epsilon| >z) }  \le  c   |a_i|,
\end{equation}
where in the last inequality we have applied  Potter's bound \eqref{eq:Potter ratio} and   fact $\sup_{j\ge 0}|a_j|<\infty$, as well as the restriction $\nu>1$.
So  putting these together we have for all $z>y_0$,
\begin{equation}\label{eq:goal R}
 \pr( R_i^* \le z  \ |\    R_i > z)\le c |a_i| .
\end{equation}

\noindent \emph{Case (II)}.

 In this case,   the tail of $S$ is regularly varying with index $-\beta>-\nu$, and the choice $\delta\in(0,1)$  always ensures $\beta<\nu+1-\delta$. By Lemma \ref{Lem:moment RV}, the truncated moment in the first term of  \eqref{eq:two term S} satisfies  for $z>y_0$ that
\begin{equation}\label{eq:S trunc moment}
\E[|S|^{\nu+1-\delta}  ;|S|\le z ]\le c z^{\nu+1-\delta}\pr(|S|>z).\end{equation}
By   \eqref{eq:R tail II}, we have    $\pr(R >z)\sim  c  \pr(|S|>z)$ as $z\rightarrow\infty$.
 Combining these facts,  one has for $z>y_0$ that
\begin{align*}
A_i(z)\le  c|a_i| .
\end{align*}
On the other hand, by  Breiman's Lemma, $\pr(|S\epsilon|>z)\sim \E|\epsilon|^{\beta} \pr(|S|>z)$ as $z\rightarrow\infty$.  Hence arguing similarly  as  Case (I) above,   we have for $z>y_0$,
\begin{equation*}
B_i(z)\le c  \frac{\pr\left(|a_i S|> z/2  \right)}{ \pr( |S| >z)}  \le  c  |a_i|,
\end{equation*}
 So \eqref{eq:goal R} holds in this case as well.

\noindent \emph{Case (III)}.

Now $\beta=\nu$. We claim that there exists a constant $c>0$ such that
\begin{align}\label{eq:R S eps}
\pr(R>z)\ge c \pr(|S\epsilon|>z)
\end{align}
for all $z>0$. We prove this below.
First we consider the case both $A_1,A_2>0$.
In view of \eqref{eq:tail linear two sides}, for some small enough constant $c>0$, we have $\pr(X_+>z)\ge c \pr(|\epsilon|>z)$ and $\pr(X_->z)\ge c \pr(|\epsilon|>z)$ for all $z>0$. This by independence implies that $\pr(S_+X_+>z)\ge c \pr(S_+|\epsilon|>z)$ and $\pr(S_-X_->z)\ge c \pr(S_-|\epsilon|>z)$. So
\[\pr(R>z)=\pr(S_+X_+>z)+\pr(S_-X_->z)\ge c \pr((S_++S_-) |\epsilon|>z)=c \pr( |S\epsilon|>z). \]
Now consider the case $A_2=0$ and $q>0$. The other case where $A_1=0$ and $q<1$ is similar and will be omitted.   Since $A_2=0$ implies $A_1>0$ and $q>0$, we have  $
  \pr( X_+ > z) \ge c\pr( |\epsilon| >z)
$ and $\pr(S_+   >z)\ge c\pr(|S|   >z)$  for some small enough constant $c>0$. Hence by independence, we have
$
  \pr(S_+X_+ > z) \ge c\pr(S_+ |\epsilon| >z)
$ and $\pr(S_+ |\epsilon| >z)\ge c\pr(|S   \epsilon| >z)$. Therefore, for all $z>0$,
\[
\pr(R>z)\ge  \pr(S_+X_+>z)   \ge  c\pr(|S \epsilon|>z).
\]
Hence \eqref{eq:R S eps} is concluded.

 Now  in view of \eqref{eq:two term S},  \eqref{eq:S trunc moment}  and \eqref{eq:R S eps}, for $z>y_0$,
\begin{equation}\label{eq:A_i(z) bound alpha=beta>1}
A_i(z)\le  c |a_i|  \frac{\pr(|S| >z)}{\pr(|S\epsilon|>z)}\le c |a_i|,
\end{equation}
where the last inequality follows from  $\pr(|S \epsilon|>z)\ge \pr(|\epsilon|\ge 1) \pr(|S|>z)$.
In addition, again by \eqref{eq:R S eps}, for $z>0$, we have
\begin{align}\label{eq:B_i(z) bound nu beta equal}
B_i(z)  \le c \frac{\pr\left(  |a_iS \epsilon|> z/2   \right) }{\pr\left(  |S \epsilon|> z   \right)}.
\end{align}
By  \cite[COROLLARY of Theorem 3]{embrechts1980closure}, $\pr\left(  |S \epsilon|> z   \right)$ is regularly varying with index $-\nu=-\beta<-1$ as $z\rightarrow\infty$. So by Potter's bound \eqref{eq:Potter ratio}, when $z>y_0$,
\[
B_i(z)\le c|a_i|.
\]
So \eqref{eq:goal R} holds in Case (III) as well.

\noindent $\bullet$ Suppose $\nu\le 1$.

Start  as   the case $\nu>1$ until the step before  \eqref{eq:two term S}. Note that now \eqref{eq:eps abs moment} may  not be applicable since $\E|\epsilon|$ is possibly infinite. Instead, applying $|u-v|\le |u|+|v|$, we bound the last two lines   above \eqref{eq:two term S}   by
\begin{align}\label{eq:two term small alpha}
&E_i(z)+ F_i(z)\notag:=\\& c|a_i| z^{-\nu-1+\delta} \E\left[   |\epsilon  | | S|^{\nu+1-\delta};\, |a_i\epsilon S|\le z, \,   |S|\le z \right]+c |a_i|  \E\left[   |\epsilon| ;\, |a_i\epsilon  S|\le z , \,    |S|> z \right],
\end{align}
where we fix $\delta\in (0,\nu)$  to be specified later.

By  Lemma \ref{Lem:moment RV},   Potter's bound \eqref{eq:Potter upper} and the fact $\sup_{j\ge 0}|a_j|<\infty$, with fixed $\eta\in(0,\nu)$, we have for all $|s|\in(0,z]$, $i\ge 0$ and $z>0$ that \[
\E\left[|\epsilon|; |\epsilon|\le z/(s|a_i|) \right]\le c (z s^{-1} |a_i|^{-1})^{1- \eta}.
\] Hence by independence and integrating out the randomness of $\epsilon$, for $z>0$,
\begin{align}\label{eq:E_i bound 1}
E_i(z)\le c |a_i|^{\eta} z^{-\nu+\delta-\eta}\E\left[  |S|^{ \nu-\delta+\eta}; \ |S|\le z    \right].
\end{align}
On the other hand, by independence,  Lemma \ref{Lem:moment RV} and Potter's bound \eqref{eq:Potter upper}, we have for $z>0$,
\begin{align}\label{eq: F_i bound}
F_i(z)\le c |a_i| \E[|\epsilon|; \  |\epsilon|\le |a_i|^{-1}]\pr(|S|>z)\le c |a_i|^{\eta}\pr(|S|>z).
\end{align}

\noindent \emph{Case (I)}.

Now $\nu<\beta$.
In this case, choose  $\eta\in (\delta,\nu)$ but sufficiently close to $\delta$, so that $\nu+\eta-\delta\le \beta $.  Then     $\E |S|^{\nu+\eta-\delta} <\infty$.   The bound \eqref{eq:E_i bound 1} simplifies to
\[
E_i(z)\le c |a_i|^{\eta} z^{-\nu+\delta-\eta}.
\]
Note, on the other hand,   that $\pr(R>z)\ge c z^{-\nu+\delta-\eta}$ for $z> y_0$ due to the  regular variation of $\pr(R>z)$ with index $-\nu$  in   \eqref{eq:R tail I} and Potter's bound \eqref{eq:Potter lower} since $\delta-\eta<0$.
Combining these above with \eqref{eq: F_i bound} and Markov inequality  $\pr(|S|>z)\le c z^{-\beta}\le c z^{-\nu+\delta-\eta}$ when $z>y_0>1$, we have for $z>y_0$
\[
A_i(z)\le c |a_i|^{\eta}.
\]
It follows from a similar argument   as  \eqref{eq:B_i bound 1<nu<beta} using Potter's bound \eqref{eq:Potter ratio}   that   for $z>y_0$,
\[
B_i(z)\le c|a_i|^{\eta}.
\]
Hence   for $z>y_0$,
\begin{equation}\label{eq:R alpha<1 goal}
\pr( R_i^* \le z  \ |\    R_i > z)\le c |a_i|^{\eta} .
\end{equation}
The conclusion follows by noting that $\delta$ and $\eta$ can be chosen  arbitrarily close to $\nu$.

\noindent \emph{Case (II)}.

Now $\beta<\nu$. Choosing again $0<\delta<\eta<\nu$, the bound \eqref{eq:E_i bound 1} in view of Lemma \ref{Lem:moment RV}  becomes for $z>y_0$,
\[
E_i(z) \le c |a_i|^{\eta}\pr(|S|>z).
\]
This time  $\pr(R>z)\sim c \pr(|S|>z)$ and $\pr(|S\epsilon|>z)\sim c \pr(|S|>z)$ as $z\rightarrow\infty$ in view of \eqref{eq:R tail II} and Breiman's Lemma respectively. Combining these with \eqref{eq: F_i bound}, we can deduce  the bound  $c |a_i|^{\eta}$ for  $A_i(z)$ when $z>y_0$.  The same bound   for  $B_i(z)$ follows similarly as the case $\nu>1$. So
\eqref{eq:R alpha<1 goal} holds.

 \noindent \emph{Case (III)}.

For this case we work with a bound different from \eqref{eq:E_i bound 1}.   Decompose the expectation in $E_i(z)$ in \eqref{eq:two term small alpha} into   $|\epsilon|\le 1$ and $|\epsilon|>1$ parts. Then  drop the restriction $|a_i\epsilon S|\le z$ in the part with $\epsilon\le 1$. Drop the restriction $|S|\le z$ and apply the inequality $|\epsilon|\le |\epsilon|^{\nu+1-\delta}$ in the part with $|\epsilon|>1$.   We then have the bound
\begin{align*}
E_i(z)&\le  c|a_i| z^{-\nu-1+\delta}\E[|S|^{\nu+1-\delta};\ |S|\le z]+ c |a_i|z^{-\nu-1+\delta} \E[|\epsilon S|^{\nu+1-\delta};\  |\epsilon S| \le z/|a_i|].
\end{align*}
 Both $\pr(|S|>z)$   and  $\pr(|\epsilon S|>z)$  (\cite{embrechts1980closure}) are regularly varying with index $-\nu=-\beta>-\nu-1+\delta$ as $z\rightarrow\infty$.  So by Lemma \ref{Lem:moment RV}   and Potter's bound \eqref{eq:Potter ratio}, with a fixed $\eta \in(0,\delta)$, we have for any $z>y_0$ that
\begin{align}\label{eq:E_i bound 2}
E_i(z)&\le    c|a_i| \pr(|S|>z)+ c |a_i|^{\delta-\nu}  \pr(|\epsilon S| > z/|a_i|)\notag\\
&\le c|a_i| \pr(|S|>z)+ c |a_i|^{ \eta} \pr(|\epsilon S| > z ).
\end{align}
 We have also $\pr(|S| >z)\pr(|\epsilon|\ge 1)\le  \pr(|S\epsilon|>z)$, and $\pr(R>z)\ge c \pr(|S\epsilon| >z)$ as in \eqref{eq:R S eps}, both of which hold   for all $z>0$. Combining these with \eqref{eq:E_i bound 2}  and  \eqref{eq: F_i bound} (choose  $\eta\in (0,\min(\delta,\nu))$), the bound   $c |a_i|^{\eta}$ holds for both $A_i(z)$ and $B_i(z)$ when $z>y_0$.
Hence for $z>y_0$,
\[
\pr( R_i^* \le z  \ |\    R_i > z)\le c |a_i|^{\eta}.
\]
The conclusion follows since $\delta$  and   $\eta$   can be chosen arbitrarily close to $\nu$.
\end{proof}

\begin{Cor}\label{Cor:TAS cond stoch vol}
Suppose the assumptions of Theorem \ref{Thm:Vol} holds.
Then the conclusion of Corollary \ref{Cor:TAS Linear} continues to hold for the stochastic volatility  extension \eqref{eq:R_i}.
\end{Cor}

\begin{Rem}
Remark \ref{Rem:conjecture} on the possibility of relaxing the restriction \eqref{eq:summability} also applies to the stochastic volatility type model $(R_i)$.
\end{Rem}

\begin{Rem}
The causal representation \eqref{eqn:gen proc} covers a wide class of  nonlinear time series models
beyond the stochastic volatility type models considered in this section, including
GARCH, autoregression with random coefficients, nonlinear autoregression, bilinear models, etc.   See,  for instance, Section 3 of \citet{liu2009strong}.  The verification of the TAS condition for these models requires nontrivial extensions  and is left for future works.
\end{Rem}

\section{The Max-Linear Extension: A Revisit}\label{sec:mmrevisit}
In this section, we revisit the max-linear extension that replaces the additive structure in (\ref{eq:linear proc}) by its maximal counterpart. \citet{Davis:Resnick:1989} presented a max-ARMA process that extends the usual additive ARMA process to its extreme-value counterpart. \citet{Hall:Peng:Yao:2002} considered the class of infinite-order moving-maximum processes, and showed that they are dense in the class of stationary processes whose finite-dimensional distributions are extreme-value of a given type. As commented in \citet{ZhangZ:2021}, the additive structure in traditional time series models cannot describe the extremal clusters and tail dependence satisfactorily in many applications, and it seems desirable to consider their non-additive extensions such as the max-linear process. \citet{Zhang:2021} studied the implication of the $\mathrm{TAS}_q$ condition on the moving-maximum process of \citet{Hall:Peng:Yao:2002} when the innovation distribution is Fr\'{e}chet, and we shall here extend their results to the case when the innovations are from a general non-negative regularly varying distribution. In particular, let $(\epsilon_j)_{j\in \bb{Z}}$ be i.i.d.\ non-negative random variables with regularly varying tail:
\begin{equation}\label{eq:max linear innovation tail}
 \pr(\epsilon_0>  x) =x^{-\nu}\ell(x)
\end{equation}
for some  function $\ell$ slowly varying at $+\infty$ and $\nu>0$. Let $\{a_j\}_{j\ge 0}$ be non-negative coefficients such that
\begin{equation}\label{eq:max linear a_j summability}
\sum_{j=0}^\infty  a_j ^{\nu'}<\infty,
\end{equation}
for some $\nu'\in (0,\nu)$.
Then as shown in   \citet{hsing1986extreme}, the moving-maximum process
\begin{equation}\label{eq:max linear proc}
X_i=\bigvee_{j=0}^\infty a_j\epsilon_{i-j}
\end{equation}
is a.s.\ finite and
\begin{equation}\label{eq:max linear tail}
\lim_{x\rightarrow\infty}\frac{\pr(X_0>x)}{\pr(\epsilon_0>x)}= \sum_{j=0}^\infty a_j^{\nu}.
\end{equation}

Following Section \ref{sec:TAS}, let $\epsilon_0^*$ be an i.i.d.\ copy of $\epsilon_0$ which is independent of $(\epsilon_j)$. Define $X_i^*$ as $X_i$ except that $\epsilon_0$ is replaced by $\epsilon_0^*$. Introduce
\[
Y_i=\bigvee_{j\ge 0, \ j\neq i}^\infty a_j \epsilon_{i-j}.
\]
Then $X_i=Y_i \vee (a_i\epsilon_0)$ and $X_i^*=Y_i \vee (a_i\epsilon_0^*)$.
So by \eqref{eq:max linear tail} and Potter's bound \eqref{eq:Potter ratio}, there exists $y_0>0$ which does not depend on $i$, such that for $z\ge y_0$,
\begin{align}
\pr(X_i^*\le z| X_i>z)&=\frac{\pr(Y_i\le z) \pr( a_i\epsilon_0^*\le z) \pr(a_i\epsilon_0> z  )}{\pr(X_0>z)}\le \frac{\pr(a_i\epsilon_0> z  )}{\pr(X_0>z)}\notag
\\&\le  c a_i^{\eta},\label{eq:max linear bound}
\end{align}
where $\eta>0$ can be fixed arbitrarily close to $\nu$. Hence we have proved the following.
\begin{Cor}\label{Cor:TAS cond max linear}
 $\mathrm{TAS}_q$ condition \eqref{eq:TAS}  holds for the moving-maximum process \eqref{eq:max linear proc}    if $\sum_{i\ge 0}  a_i ^{\eta/q}<\infty$ for some $\eta\in (0,\nu)$.
\end{Cor}
\begin{Rem}
It is possible to slightly improve \eqref{eq:max linear bound} for certain slowly varying function $\ell(x)$ in \eqref{eq:max linear innovation tail}. For example in \citet{Zhang:2021}, the bound \eqref{eq:max linear bound} can be strengthened to $ca_i^{\nu}$ for Fr\'echet distribution which leads to the sufficient condition $\sum_{i\ge 0}  a_i ^{\nu/q}<\infty$ for $\mathrm{TAS}_q$. Similar improvements can also be considered for Theorems \ref{Thm:linear} and \ref{Thm:Vol}. We do not pursue such a refinement here since it does not lead to a substantial  statistical consequence.
\end{Rem}

\section{Extensions via Monotone Transforms}\label{sec:mono}
Recall a process $X=(X_i)$ given by  \eqref{eqn:gen proc}  satisfies
 the $\mathrm{TAS}_q$   condition  if $\lim_{y \uparrow \mathcal{U}_X} \Theta_{y,q}^{(X)} < \infty$ (cf.\ \eqref{eq:TAS}). Under the $\mathrm{TAS}_q$ condition, there exists a real \[x^*_q=\inf\{y<\mathcal{U}_X:\ \Theta_{y,q}^{(X)} < \infty\}<\mathcal{U}_X.\]  Note that $  \Theta_{y,q}^{(X)} <\infty$ for any $y>x^*_q$. The following proposition provides sufficient conditions for $\mathrm{TAS}_q$ to carry over through monotonic transforms.
\begin{proposition}\label{Pro:monotone}
Suppose a stationary process $X=(X_i)$  is given by the model \eqref{eqn:gen proc}, whose marginal distribution has   lower and upper end points  $\mathcal{L}_X=\sup\{x\in \mathbb{R}:\ \pr(X_0\ge x)=0\}$ and $\mathcal{U}_X=\inf\{x\in \mathbb{R}:\ \pr(X_0\le x)=1\}$ respectively.
 Suppose  $X$ satisfies
 the $\mathrm{TAS}_q$   condition, $q>0$.
 
Let $K:[\mathcal{L}_X,\mathcal{U}_X]\mapsto [-\infty,\infty]$ be a non-decreasing function. Suppose  the transformed stationary process $Y=(Y_i)=(K(X_i))$ has marginal 
upper end point 
$\mathcal{U}_Y$    satisfying $\mathcal{U}_Y=K(\mathcal{U}_X)$.
Then $Y$ satisfies the   $\mathrm{TAS}_q$   condition  under either of the following conditions:
\begin{enumerate}[(a)]
\item The function $K$ is strictly increasing on $(x_0,\mathcal{U}_X]$ for some $x_0<\mathcal{U}_X$; 
\item  We have  $x_1:=\inf\{x\in [\mathcal{L}_X,\mathcal{U}_X]: K(x)=\mathcal{U}_Y\}>x^*_q$, and  there exists   $x_0<x_1$ such that   $P(X_0=x)=0$ for all $x\in (x_0,x_1)$.

\end{enumerate}
\end{proposition}
\begin{Rem}
Condition (a) says $K$ is ultimately strictly increasing. In Condition (b), note that $x_1=\mathcal{U}_X$ if $K(x)<\mathcal{U}_Y$ for all $x<\mathcal{U}_X$, under which $x_1>x_q^*$ always holds if $X$ satisfies $\mathrm{TAS}_q$. The second assumption in Condition (b) imposes  ultimate continuity of the marginal distribution $X_0$.

The assumption $\mathcal{U}_Y=K(\mathcal{U}_X)$ is made without loss of generality. In general, it is possible that $K(\mathcal{U}_X)>\mathcal{U}_Y$. But since $\pr\left(Y_0>\mathcal{U}_Y\right)=0$, one may modify the definition of $K$ by a  truncation  as  $K1_{\{ K\le  \mathcal{U}_Y\}}+\mathcal{U}_Y1_{\{ K> \mathcal{U}_Y\}}$ without changing $Y$ almost surely. 

In the case where $K$ is only defined on the open interval $(\mathcal{L}_X,\mathcal{U}_X)$ (similarly for other half-open-type intervals), one may without loss of generality extend the domain of $K$ to  $[\mathcal{L}_X,\mathcal{U}_X]$ by setting $K(\mathcal{L}_X)=\lim_{u\downarrow
 \mathcal{L}_X} K(x)$ and $K(\mathcal{U}_X)=\lim_{u\uparrow\mathcal{U}_X}K(x)$.
\end{Rem}

 \begin{proof}[Proof of Proposition \ref{Pro:monotone}]

 We follow the notation in Section \ref{sec:TAS}.
 
  \noindent(a)
 Let $K^{-1}$ denote the inverse of $K$ when the latter is restricted to $(x_0,\mathcal{U}_X]$.  With $ K  (x_0,\mathcal{U}_X]$  denoting the image of   $(x_0,\mathcal{U}_X]$ under $K$, observe that $\inf  K (x_0,\mathcal{U}_X]<\mathcal{U}_Y$. So  with $ \inf  K (x_0,\mathcal{U}_X]<y< \mathcal{U}_Y$, one has
\begin{align*}
\sup_{z\ge y} \pr(K(X_i^*) \leq z \mid K(X_i) > z)&\le  \sup_{ \ z\in  K  (x_0,\mathcal{U}_X] } \pr( X_i^* \leq K^{-1}(z) \mid X_i > K^{-1}(z))
\\&\le \sup_{u\ge  K^{-1}(x_0)} \pr( X_i^* \leq u \mid X_i > u)= \theta_{K^{-1}(x_0)}^{(X)}(i).  
\end{align*}
The conclusion follows if, without loss of generality,  $x_0$ is chosen sufficently close to $\mathcal{U}_X$   so that  $K^{-1}(x_0)>x_q^*$.

\medskip

\noindent(b)
 For  $z\in [\mathcal{L}_Y,\mathcal{U}_Y]$, we  define $I_z=\{x\in [\mathcal{L}_X, \mathcal{U}_X]:\  K(x)\le z \}$ and $J_z=\{x\in  [\mathcal{L}_X, \mathcal{U}_X]:\ K(x)> z \}$, both of which are intervals due to the monotonicity of $K$.  Set $b(z)=\sup I_z=\inf J_z$, which is non-decreasing in $z$.  Assume without loss of generality $x_0\in (x_q^*,x_1)$.  

We claim that as  $z\uparrow \mathcal{U}_Y$, we have $b(z)\uparrow  x_1$. Indeed, otherwise, there exists $x_1'<x_1$  such that  $b(z)\le x_1'$ for any $z<\mathcal{U}_Y$. Hence  $K(x)< \mathcal{U}_Y$  implies $x\le x_1'$, which contradicts with  the definition of $x_1$.

Now with the claim above, we can choose $y<\mathcal{U}_Y$  sufficiently close to $ \mathcal{U}_Y$  so that $x_1>b(y)>x_0>x_q^*$. Then applying the assumptions, we have
\begin{align*}
\sup_{z\ge y}\pr(K(X_i^*) \leq z \mid K(X_i) > z)&=\sup_{z\ge y}\pr( X_i^* \in I_z \mid  X_i\in J_z)=\sup_{z\ge y }\pr( X_i^* \le b(z)  \mid  X_i> b(z))
\\ & \le    \sup_{u\ge b(y) }\pr( X_i^* \le u  \mid  X_i> u) ,
\end{align*}
and the conclusion   follows.

\begin{example}
Consider a linear process $(X_i)$ as in \eqref{eq:linear proc}, which satisfies the $\mathrm{TAS}_q$ condition, $q>0$ (cf.\ Corollary \ref{Cor:TAS Linear} and Remark \ref{Rem:conjecture}).  Based on the assumptions made, the marginal distribution of $X_0$ is typically continuous (cf.\ the Proof of Lemma \ref{Lem:control}), and we shall assume so.  

To model integer-valued tail-dependent data, one may consider $(Y_i)=(K(X_i))=(\lfloor X_i \rfloor)$, where $K(x)=\lfloor  x \rfloor$ is the floor function (i.e., greatest integer not exceeding $x$).  Based on Proposition \ref{Pro:monotone}, in particular, applying Condition (b) (note that $x_1=\mathcal{U}_X=\infty$ in this case), the integer-valued process $(Y_i)$ also satisfies $\mathrm{TAS}_q$.

 The same consideration applies to the stochastic volatility extension in Section \ref{sec:stochasticvolatilitymodel} and the max-linear process in Section \ref{sec:mmrevisit}.
\end{example}

 \end{proof}

\section{Application: Limit Theorems in Statistical Context}\label{application}
In this section, we provide implications of the developed results on some limit theorems of tail quantities with statistical motivations.
\subsection{High Quantile Regression}\label{subsec:highquantileregression}

We first consider the high quantile regression problem studied in \citet{Zhang:2021}. Suppose we observe the $n$-th row of a triangular array which consists of response variables $U_{1,n},\ldots,U_{n,n} \in \mathbb R$ associated with a set of explanatory variables $W_{1,n},\ldots,W_{n,n} \in \mathbb R^p$ according to the quantile regression model \citep{Koenker:Bassett:1978}
\begin{equation*}
U_{i,n} = W_{i,n}^\top \beta_n + X_{i,n},
\end{equation*}
where $^\top$ denotes the transpose, $\beta_n \in \mathbb R^p$ is the regression coefficient for the $(1-\alpha_n)$-th quantile, and $X_{i,n} = U_{i,n} - W_{i,n}^\top \beta_n$ is the auxiliary variable satisfying $\pr(X_{i,n} \leq 0) = \pr(U_{i,n} \leq W_{i,n}^\top \beta_n) = 1-\alpha_n$. The quantile regression coefficient $\beta_n$ can then be estimated by the high quantile regression estimator
\begin{equation}\label{eqn:betanhat}
\hat{\beta}_n = \mathop{\mathrm{argmin}}_{\eta \in \mathbb R^p} \sum_{i=1}^n \phi_{1-\alpha_n}(U_{i,n} - W_{i,n}^\top \eta),
\end{equation}
where $\phi_{1-\alpha_n}(u) = (1-\alpha_n) u^+ + \alpha_n(-u)^+$ is the check function with $u^+ = \max(u,0)$. Compared with the traditional quantile regression \citep{Koenker:Bassett:1978,Koenker:2005}, the high quantile regression in (\ref{eqn:betanhat}) requires the quantile level $1-\alpha_n$ to approach the unit as the sample size increases to capture the tail phenomena. Assuming that the auxiliary process $(X_{i,n}) \in \mathrm{TAS}_2$ (a triangular array variant), under some mild conditions on the smoothness of the marginal distribution and the design matrix, \citet{Zhang:2021} obtained the consistency and the central limit theorem for the high quantile regression estimator (\ref{eqn:betanhat}); see Theorems 1 and 2 of \citet{Zhang:2021}.

In many applications, one is interested in estimating a high quantile of a given stationary tail dependent time series, which relates to the situation when $W_{i,n} \equiv 1$. In this case, one observes a stationary time series $(U_{i,n})=(U_i)$ whose marginal distribution is denoted by $F(u) = \pr(U_{i} \leq u)$, and (\ref{eqn:betanhat}) can still be used to obtain an estimator for the $(1-\alpha_n)$-th quantile $\beta_n$. We make the following assumption.
\begin{itemize}
\item[(Q)] There exists an $\alpha \in (0,1)$ such that $F(\cdot)$ is continuously differentiable with uniformly bounded and strictly positive derivative $f(\cdot)$ in its upper tail $\{F^{-1}(1-\alpha),F^{-1}(1)\}$ with $|F^{-1}(1) - F^{-1}(1-\alpha)| > 0$.
\end{itemize}
Assumption (Q) mostly concerns the smoothness of the underlying distribution $F(\cdot)$ in the tail part and is satisfied by many commonly used distributions. Let $X_{i,n} = U_i - \beta_n$ be the associated auxiliary variable, the following theorem provides the consistency and central limit theorem of $\hat \beta_n$, which follows from Theorems 1 and 2 of \citet{Zhang:2021},   with some of the conditions  simplified for the current intercept case.

\begin{theorem}[\citealp{Zhang:2021}]\label{thm:HQ}
Assume (Q), $(U_i) \in \mathrm{TAS}_2$, $\alpha_n \to 0$ and $n\alpha_n \to \infty$. If
\begin{equation*}
\psi_n = (n\alpha_n)^{1/2} {f_n(0) \over 1-F_n(0)} \to \infty
\end{equation*}
and
\begin{equation*}
\max_{1 \leq i \leq n} \sup_{|\eta| \leq c} \left|{f_n(\psi_n^{-1} \eta) - f_n(0) \over f_n(0)}\right| \to 0
\end{equation*}
for any $c < \infty$, then
\begin{equation*}
\hat \beta_n - \beta_n = O_p(\psi_n^{-1}).
\end{equation*}
If in addition the limit
\begin{equation*}
\rho_k = \lim_{n \to \infty} \mathrm{cor}(1_{\{X_{0,n} > 0\}},1_{\{X_{k,n} > 0\}})
\end{equation*}
exists for each $k \in \mathbb Z$ and $\sum_{k \in \mathbb Z} \rho_k > 0$, then
\begin{equation*}
\psi_n(\hat \beta_n - \beta_n) \to N\left(0,\sum_{k \in \mathbb Z} \rho_k\right).
\end{equation*}
\end{theorem}

Assumptions concerning $F(\cdot)$ and $f(\cdot)$ in the above theorem can be verified for a number of distribution functions, including the uniform, exponential, normal and Pareto distributions; see for example the discussions in \citet{Zhang:2021}. We shall in the following provide a discussion on the tail adversarial stability condition that $(U_{i}) \in \mathrm{TAS}_2$.

For the linear process (\ref{eq:linear proc}) with S$\nu$S innovations, $\nu\in (0,2)$, by the discussion in Section \ref{sec:linearprocesses}, one can show that the $\mathrm{TAS}_2$ condition needed for high quantile regression inference as in \citet{Zhang:2021} is satisfied if the coefficients
\begin{equation*}
a_i \sim ci^{-\zeta},\quad i\rightarrow\infty,
\end{equation*}
for some $\zeta > \max(2,2/\nu)$.  For more general linear regularly varying process  with index $-\nu$, $\nu>0$,  satisfying the assumptions of Theorem \ref{Thm:linear}, we need
$
\zeta >  \max (2, 1+2/\nu)$ under the power decay condition for $a_i$  above.
By Corollary \ref{Cor:TAS cond stoch vol}, this will continue to hold for the stochastic volatility extension given in (\ref{eq:R_i}) as well. As a comparison, \citet{Chernozhukov:2005} studied high quantile regression under the strong mixing framework of \citet{Rosenblatt:1956} and used an additional condition to control the joint probability of nearby tail events. Such a condition can essentially be interpreted as a negligibility condition on tail dependence, and is generally not expected to hold for processes exhibiting nonnegligible tail dependence. Therefore, the TAS framework seems to provide a convenient framework for studying high quantile regression of tail dependent time series data.

\subsection{Tail Autocorrelation Analysis}\label{subsec:tailautocorrelation}
We in this section consider the problem of tail autocorrelation analysis, which extends the traditional autocorrelation analysis to the tail setting. For this, let $x_n \to \infty$ be an extremal threshold, then the degree of tail dependence at lag $k$ can be quantified by the conditional probability $\pr(X_{1+k} > x_n \mid X_1 > x_n)$ as proposed in \citet{Zhang:2005}; see also \citet{Linton:Whang:2007} when the threshold $x_n$ is represented using quantiles. The tail autocorrelation at lag $k$ is then defined as
\begin{equation*}
\tau_{x_n}(k) = {\pr(X_{1+k} > x_n \mid X_1 > x_n) - \pr(X_{1+k} > x_n) \over 1-\pr(X_1 > x_n)},
\end{equation*}
which standardizes the conditional probability $\pr(X_{1+k} > x_n \mid X_1 > x_n)$ in the form of a correlation coefficient. \citet{Zhang:2022} established the consistency and a two-phase central limit theorem of sample tail autocorrelations under the tail adversarial stability framework, where it was assumed that
\begin{itemize}
\item[(Z1)] the underlying process $(X_i) \in \mathrm{TAS}_q$ for some $q > 4$; and
\item[(Z2)] the extremal threshold satisfies $\bar F(x_n) \to 0$ and $n\bar F(x_n) \to \infty$,
\end{itemize}
with $\bar F(x_n) = \pr(X_1 > x_n)$ being the marginal survival function. By Corollary \ref{Cor:TAS Linear} with $q>4$, condition (Z1) holds for the regularly varying linear process (\ref{eq:linear proc}) if
\begin{equation*}
\sum_{i=0}^\infty |a_i|^{\iota} < \infty
\end{equation*}
for some $\iota < \min(\nu/4,1/4)$. Condition (Z2) is very mild as $\bar F(x_n) \to 0$ only requires the threshold $x_n$ to be in the tail and $n\bar F(x_n) \to \infty$ essentially requires the amount of the data in the tail goes to infinity so that we can have the consistency without assuming any parametric assumption on the tail.

On the other hand, \citet{Davis:Mikosch:2009} considered adopting the strong mixing framework and provided a central limit theorem for sample tail autocorrelations in their Corollary 3.4 that aligns with the Phase I result of \citet{Zhang:2022}. Let $\mathcal F_{i,j} = \sigma(X_k, i \leq k \leq j)$ be the $\sigma$-field generated by $X_i,\ldots,X_j$ for $i \leq j$, it was assumed in \citet{Davis:Mikosch:2009} that
\begin{itemize}
\item[(DM1)] the underlying process $(X_i)$ is $\alpha$-mixing and the strong mixing coefficient
    \begin{equation*}
    \alpha(i) = \sup_{A \in \mathcal F_{-\infty}^k,\ B \in \mathcal F_{k+i}^\infty} |\pr(A \cap B) - \pr(A)\pr(B)|
    \end{equation*}
    satisfies
    \begin{equation*}
    \lim_{n \to \infty} m_n \sum_{i=r_n}^\infty \alpha(i) = 0
    \end{equation*}
    for some $m_n, r_n \to \infty$ with $\lim_{n \to \infty} m_n\pr(|X_1| > x_n) = 1$, $m_n/n \to 0$ and $r_n/m_n \to 0$;
\item[(DM2)] for all $\varpi > 0$,
    \begin{equation*}
    \lim_{k \to \infty} \limsup_{n \to \infty} m_n \sum_{i=k}^{r_n} \pr(|X_i| > \varpi x_n, |X_0| > \varpi x_n) = 0;
    \end{equation*}
    and
\item[(DM3)] $n\alpha_{r_n}/m_n \to 0$ and $m_n = o(n^{1/3})$, where $m_n = o(n^{1/3})$ can be replaced by
    \begin{equation*}
    {m_n^4 \over n} \sum_{i=r_n}^{m_n} \alpha(i) \to 0\ \mathrm{and}\ {m_n r_n^3 \over n} \to 0.
    \end{equation*}
\end{itemize}
We shall here verify conditions (DM1)--(DM3) for the regularly varying linear process (\ref{eq:linear proc}). \citet{Davis:Mikosch:2009} considered the special case of a finite-order ARMA model, for which the coefficient $a_i$ in its linear representation follows a geometric decay. In this case, the strong mixing coefficient also follows a geometric decay, which largely simplified the verification of conditions (DM1)--(DM3). As before, we assume that
\begin{equation*}
a_i \sim ci^{-\zeta},\quad i\rightarrow\infty,
\end{equation*}
for some $\zeta > 0$. For simplicity of illustration, we also assume that the regularly varying innovations in (\ref{eq:linear proc}) satisfies \[\nu = 1\] and $\ell(x) \to 1$ as $x \to \infty$ in (\ref{eq:abs RV}).
On the other hand, obtaining a sharp estimate of the strong mixing coefficient is highly nontrivial. The best estimate we can find in literature is     Lemma 15.3.1 of \citet{kulik2020heavy}  adapted from the results of \citet{pham1985some}. Specifically, assuming that the index $\zeta > 3$, the strong mixing coefficient has the bound
\begin{equation}\label{eqn:alphanbound}
\alpha(n) = O\{n^{-(\zeta - 1)(1 - \varepsilon)/(2-\varepsilon) + 1}\},
\end{equation}
where $\varepsilon \in (0,1)$ is a constant that can be taken arbitrarily small. Note that $m_n \sim 1/\bar F(x_n)$, condition (DM1) is satisfied if $\bar F(x_n) \to 0$, $n\bar F(x_n) \to \infty$, $r_n \to \infty$, and
\begin{equation}\label{eqn:rncondition}
r_n\bar F(x_n) + r_n^{2-(\zeta-1)(1-\varepsilon)/(2-\varepsilon)}/\bar F(x_n) \to 0.
\end{equation}
In addition, by a similar argument used in (15.3.33) of \citet{kulik2020heavy}, condition (DM2) is satisfied if
\begin{equation*}
\sum_{i=0}^\infty i|a_i|^\varsigma < \infty
\end{equation*}
for some $\varsigma \in (0,1)$. Since $a_i \sim ci^{-\zeta}$, there exists a compatible $r_n \to \infty$ such that (DM1) and (DM2) are satisfied if $\zeta - 1 > 1$ and
\begin{equation*}
2-(\zeta-1)(1-\varepsilon)/(2-\varepsilon) < -1.
\end{equation*}
By choosing $\varepsilon > 0$ arbitrarily small, the above indicates that $\zeta > 7$. In contrast, condition (Z1) from the tail adversarial stability framework only requires that $\zeta > 4$. It is remarkable that the strong mixing framework requires an additional condition (DM3), which typically leads to more restrictive conditions on how extremal the tail can be. For example, the condition $m_n = o(n^{1/3})$ in (DM3) requires that $n\{\bar F(x_n)\}^3 \to \infty$, while in comparison condition (Z2) from the tail adversarial stability framework only requires that $n \bar F(x_n) \to \infty$. Note that the condition $m_n = o(n^{1/3})$ in (DM3) can be replaced by its alternative $(m_n^4/n) \sum_{i=r_n}^{m_n} \alpha(i) \to 0$ and $m_n r_n^3/n \to 0$, for which by (\ref{eqn:alphanbound}) and Karamata's theorem it suffices to have
\begin{equation*}
{r_n^{2-(\zeta-1)(1-\varepsilon)/(2-\varepsilon)} \over n\{\bar F(x_n)\}^4} \to 0\ \mathrm{and}\ {r_n^3 \over n \bar F(x_n)} \to 0.
\end{equation*}
This, together with (\ref{eqn:rncondition}), make it difficult to work out the actual condition as it depends on the nontrivial interplay between how extremal the tail can be and how fast the linear coefficients decay to zero. Since $r_n\bar F(x_n) \to 0$ and $r_n^{(\zeta-1)(1-\varepsilon)/(2-\varepsilon)-2}\bar F(x_n) \to \infty$ by (\ref{eqn:rncondition}), it is then necessary, though probably not sufficient, to have
\begin{equation*}
n\{\bar F(x_n)\}^{6-(\zeta-1)(1-\varepsilon)/(2-\varepsilon)} \to \infty\ \mathrm{and}\ n\{\bar F(x_n)\}^{1+3/\{(\zeta-1)(1-\varepsilon)/(2-\varepsilon)-2\}} \to \infty,
\end{equation*}
which is still stronger than condition (Z2) from the tail adversarial stability framework. We also remark that the condition $n\alpha_{r_n}/m_n \sim n\bar F(x_n)\alpha_{r_n} \to 0$ in (DM3) prevents $\bar F(x_n)$ from going to zero too slowly, while condition (Z2) only requires that $\bar F(x_n) \to 0$. Therefore, in addition to being more tractable, the tail adversarial stability framework can also lead to cleaner and weaker conditions on not only how strong the tail dependence can be but also how extremal the tail can be.

\subsection{Tail Empirical Distribution}\label{subsec:tailempiricaldistribution}
We in this section consider estimating the tail probability $T(x_n): = \bar{F}(x_n) = \pr(X_i > x_n)$ by its empirical version
\begin{equation*}
\widehat{T}(x_n) = n^{-1} \sum_{i=1}^n 1_{\{X_i > x_n\}}
\end{equation*}
when the threshold $x_n \to \infty$ satisfying also $\E[n \widehat{T}(x_n)]  =n \bar{F}(x_n)\rightarrow\infty$. For simplicity we assume $T(x)\sim c x^{-\nu}$ as $x\rightarrow\infty$, and hence the aforementioned condition becomes $x_n\ll n^{1/\nu}$ as $n\rightarrow\infty$ (recall we write $a_n\ll b_n$ if $a_n=o(b_n)$).  To understand the convergence rate and the associated asymptotic distribution, it requires a limit theorem on the difference
\begin{equation}\label{eq:tail emp}
\widehat{T}(x_n) - T(x_n) = n^{-1} \sum_{i=1}^n [1_{\{X_i > x_n\}} - T(x_n)].
\end{equation}
For this, by the proof of Theorem 2 in \citet{Zhang:2021}, one can show that the central limit theorem
\begin{equation*}
\left[{n \over T(x_n)\{ 1-T(x_n)\}}\right]^{1/2} \{\widehat{T}(x_n) -  T(x_n)\} \to_d N(0,\sigma^2)
\end{equation*}
as $n\rightarrow\infty$ holds for some $\sigma^2 > 0$ if the process $(X_i) \in \mathrm{TAS}_2$ along with some other mild regularity conditions. \citet{Rootzen:2009} applied the $\beta$-mixing condition and obtained a weak convergence result for the tail empirical process (introducing an  additional  parameter into \eqref{eq:tail emp})  which implies the above central limit theorem; see also \cite[Chapter 9]{kulik2020heavy}.      We   leave  a full development of functional limit theorem for tail empirical process under the tail adversarial stability framework as a future work, and restrict the discussion on the marginal central limit theorem.

  We shall   make a comparison between the TAS framework and the $\beta$-mixing framework described in \cite[Section 9.2.3]{kulik2020heavy} for heavy-tailed linear processes  (\ref{eq:linear proc}).  Assume as before that the linear process coefficients satisfy for some $\zeta>0$ that
  \begin{equation*}
a_i \sim ci^{-\zeta},\quad i\rightarrow\infty.
\end{equation*}
Below for simplicity, we  focus only on the implications on the exponent $\zeta$ and the threshold $x_n$ and omit some additional technical assumptions involved.
 For the TAS framework,   as in Section \ref{subsec:highquantileregression}, for a linear process with S$\nu$S innovations, $\nu\in (0,2)$, the process is $\mathrm{TAS}_2$ if   $\zeta > \max(2,2/\nu)$; for the more general regularly varying linear processes satisfying the assumptions of Theorem \ref{Thm:linear}, we need the stronger restriction $\zeta>\max(2, 1+2/\nu )$.    On the other hand,     to establish the central limit theorem  under the $\beta$-mixing framework as described in  \cite[Proposition 9.2.5]{kulik2020heavy}, one needs  the conditions denoted as $\mathcal{R}(r_n,x_n)$, $\beta(r_n,\ell_n)$ and $\mathcal{S}(r_n,x_n)$, where $r_n$ and $\ell_n$ are two sequences tending to infinity such that  $\ell_n \ll r_n\ll n$ as $n\rightarrow\infty$.   First, in view of \cite[Section 15.3]{kulik2020heavy}, the $\beta$-mixing condition is satisfied if $\zeta > 2+1/\nu$ with a $\beta$-mixing coefficient estimate $\beta_n=O(n^{1-(\zeta-1)(\nu-\varepsilon)/(1+\nu-\varepsilon)})=o(1)$ as $n\rightarrow\infty$, where $\varepsilon>0$ can be chosen arbitrarily small.  Now the conditions $\mathcal{R}(r_n,x_n)$ and $\beta(r_n,\ell_n)$ respectively require:
 \begin{equation} \label{eq:ell_n r_n n}
 r_n^{1/\nu} \ll x_n\ll n^{1/\nu} \quad \text{ and  }  \quad n \ell_n^{1-(\zeta-1)(\nu-\varepsilon)/(1+\nu-\varepsilon)}\ll r_n .
 \end{equation}
 According to \cite[Section 15.13]{kulik2020heavy}, the condition $\mathcal{S}(r_n,x_n)$ holds when $\zeta>\max(2/\nu,1)$.
  As a summary, the $\beta$-mixing framework minimally requires $\zeta>\max(2/\nu, 2+1/\nu)$, which is  more stringent than   the TAS requirement for S$\nu$S innovations when  $\nu \in (1/2,2)$, and more stringent than the TAS requirement for    general regularly varying innovations when $\nu>1$. The $\beta$-mixing framework also introduces a lower boundary  rate for the threshold $x_n$ which is not present in the TAS framework:  \eqref{eq:ell_n r_n n} and $\ell_n\ll r_n$ together imply that  $x_n\gg n^{(1+\nu)/ ((\zeta-1)\nu^2)} $.

We also consider the moving-maximum process (\ref{eq:max linear proc}). Assume for simplicity that the innovations are $\nu$-Fr\'echet and  again
the coefficients
\begin{equation*}
a_i \sim ci^{-\zeta}, \quad i\rightarrow\infty.
\end{equation*}  In view of Section \ref{sec:mmrevisit}, the process is $\mathrm{TAS}_2$ if
 $\zeta > 2/\nu$.  On the other hand,  by \cite[Theorems 13.4 and 13.5]{kulik2020heavy}, the $\beta$-mixing condition holds    holds  if $\zeta>3/\nu$ (more stringent than the TAS requirement) with a beta mixing coefficient estimate $\beta_n=O(n^{3-\zeta\nu})=o(1)$ as $n\rightarrow\infty$, and  the condition $\mathcal{S}(r_n,x_n)$  mentioned above also follows. The conditions $\mathcal{R}(r_n,x_n)$ and $\beta(r_n,\ell_n)$ mentioned above respectively require
  \begin{equation*}
 r_n^{1/\nu} \ll x_n\ll n^{1/\nu} \quad \text{ and  }  \quad n \ell_n^{3-\zeta\nu}\ll r_n,
 \end{equation*}
which as above imply a lower rate restriction for the threshold: $x_n\gg n^{1/(\nu(\zeta\nu-2))} $, which is not present in the TAS case.

\section{Conclusion}\label{sec:conclusion}
Although various tail dependence measures have been proposed to summarize the degree of the underlying tail dependence, few is useful for developing limit theorems of tail dependent time series. Because of this limitation on available tools, the existing literature to date still largely relies on the strong mixing condition of \citet{Rosenblatt:1956} to obtain limit theorems of tail dependent time series. However, the strong mixing condition of \citet{Rosenblatt:1956} was not originally developed to handle dependence in the tail, and as a result additional conditions that control more specifically the degree of dependence in the tail are often needed together with the strong mixing condition. Such conditions can lead to either additional restrictions on the strong mixing coefficient that cannot be easily made explicit or conditions that cannot be fully captured by the strong mixing coefficient. In addition, the supreme over two sigma algebras makes it generally a difficult task to derive a sharp estimate  of the strong mixing coefficient. Recently, \citet{Zhang:2021} proposed an alternative framework based on a new notion of tail adversarial stability, which has been shown to be useful in obtaining nontrivial limit theorems of tail dependent time series. The advantage over the classical strong mixing framework was illustrated in \citet{Zhang:2021} for the moving-maximum process of \citet{Hall:Peng:Yao:2002}. This article studies the tail adversarial stability for the class of regularly varying additive linear processes, which has also been   adopted in modeling extremal clusters and tail dependence in time series. It can be seen from our main results in Section \ref{sec:linearprocesses} that the tail adversarial stability condition can be translated into   mild conditions on the linear coefficients, which can be   weaker than those under the strong mixing framework; see for example the discussion in Section \ref{application}. Extensions to the stochastic volatility model and the max-linear processes are also considered.

%
%

\bibliographystyle{BibliographyStyleTing}
\setlength{\bibsep}{0mm}
\bibliography{BibliographyTing,BibBai}

\begin{thebibliography}{52}
\providecommand{\natexlab}[1]{#1}

\bibitem[Balla et~al., 2014]{Balla:Ergen:Migueis:2014}
\textsc{Balla, E., Ergen, I. and Migueis, M.} (2014).
\newblock Tail dependence and indicators of systemic risk for large us
  depositories.
\newblock \emph{Journal of Financial Stability}, \textbf{15}, 195--209.

\bibitem[Barbe and McCormick, 2009]{barbe2009asymptotic}
\textsc{Barbe, P. and McCormick, W.~P.} (2009).
\newblock \emph{Asymptotic expansions for infinite weighted convolutions of
  heavy tail distributions and applications}.
\newblock American Mathematical Soc.

\bibitem[Basrak and Segers, 2009]{basrak2009regularly}
\textsc{Basrak, B. and Segers, J.} (2009).
\newblock Regularly varying multivariate time series.
\newblock \emph{Stochastic Processes and their Applications}, \textbf{119},
  1055--1080.

\bibitem[Bingham et~al., 1989]{bingham:goldie:teugels:1989:regular}
\textsc{Bingham, N., Goldie, C. and Teugels, J.} (1989).
\newblock \emph{Regular Variation}.
\newblock Encyclopedia of Mathematics and Its Applications. Cambridge
  University Press.

\bibitem[Bogachev, 2007]{bogachev2007measure}
\textsc{Bogachev, V.~I.} (2007).
\newblock \emph{Measure Theory}, vol.~1.
\newblock Springer.

\bibitem[Chernozhukov, 2005]{Chernozhukov:2005}
\textsc{Chernozhukov, V.} (2005).
\newblock Extremal quantile regression.
\newblock \emph{The Annals of Statistics}, \textbf{33}, 806--839.

\bibitem[Chernozhukov and Fern\'{a}ndez-Val,
  2011]{Chernozhukov:FernandezVal:2011}
\textsc{Chernozhukov, V. and Fern\'{a}ndez-Val, I.} (2011).
\newblock Inference for extremal conditional quantile models, with an
  application to market and birthweight risks.
\newblock \emph{The Review of Economic Studies}, \textbf{78}, 559--589.

\bibitem[Coles et~al., 1999]{Coles:Heffernan:Tawn:1999}
\textsc{Coles, S., Heffernan, J. and Tawn, J.} (1999).
\newblock Dependence measures for extreme value analyses.
\newblock \emph{Extremes}, \textbf{2}, 339--365.

\bibitem[Davis and Mikosch, 2009]{Davis:Mikosch:2009}
\textsc{Davis, R.~A. and Mikosch, T.} (2009).
\newblock The extremogram: A correlogram for extreme events.
\newblock \emph{Bernoulli}, \textbf{15}, 977--1009.

\bibitem[Davis et~al., 2012]{Davis:Mikosch:Cribben:2012}
\textsc{Davis, R.~A., Mikosch, T. and Cribben, I.} (2012).
\newblock Towards estimating extremal serial dependence via the bootstrapped
  extremogram.
\newblock \emph{Journal of Econometrics}, \textbf{170}, 142--152.

\bibitem[Davis and Resnick, 1989]{Davis:Resnick:1989}
\textsc{Davis, R.~A. and Resnick, S.~I.} (1989).
\newblock Basic properties and prediction of max-{ARMA} processes.
\newblock \emph{Advances in Applied Probability}, \textbf{21}, 781--803.

\bibitem[{de Haan} and Resnick, 1977]{deHaan:Resnick:1977}
\textsc{{de Haan}, L. and Resnick, S.~I.} (1977).
\newblock Limit theory for multivariate sample extremes.
\newblock \emph{Zeitschrift f\"{u}r Wahrscheinlichkeitstheorie und Verwandte
  Gebiete}, \textbf{40}, 317--337.

\bibitem[Draisma et~al., 2004]{Draisma:Drees:Ferreira:DeHaan:2004}
\textsc{Draisma, G., Drees, H., Ferreira, A. and {De Haan}, L.} (2004).
\newblock Bivariate tail estimation: dependence in asymptotic independence.
\newblock \emph{Bernoulli}, \textbf{10}, 251--280.

\bibitem[Drees, 2003]{Drees:2003}
\textsc{Drees, H.} (2003).
\newblock Extreme quantile estimation for dependent data, with applications to
  finance.
\newblock \emph{Bernoulli}, \textbf{9}, 617--657.

\bibitem[Embrechts and Goldie, 1980]{embrechts1980closure}
\textsc{Embrechts, P. and Goldie, C.~M.} (1980).
\newblock On closure and factorization properties of subexponential and related
  distributions.
\newblock \emph{Journal of the Australian Mathematical Society}, \textbf{29},
  243--256.

\bibitem[Embrechts et~al., 2002]{Embrechts:Mcneil:Straumann:2002}
\textsc{Embrechts, P., Mcneil, A. and Straumann, D.} (2002).
\newblock Correlation and dependence in risk management: Properties and
  pitfalls.
\newblock In M.~A.~H. Dempster, editor, \emph{Risk Management: Value at Risk
  and Beyond}. Cambridge University Press, pp. 176--223.

\bibitem[Ferro and Segers, 2003]{Ferro:Segers:2003}
\textsc{Ferro, C. A.~T. and Segers, J.} (2003).
\newblock Inference for clusters of extreme values.
\newblock \emph{Journal of the Royal Statistical Society: Series B (Statistical
  Methodology)}, \textbf{65}, 545--556.

\bibitem[Hall et~al., 2002]{Hall:Peng:Yao:2002}
\textsc{Hall, P., Peng, L. and Yao, Q.} (2002).
\newblock Moving-maximum models for extrema of time series.
\newblock \emph{Journal of Statistical Planning and Inference}, \textbf{103},
  51--63.

\bibitem[Hill, 2009]{Hill:2009}
\textsc{Hill, J.~B.} (2009).
\newblock On functional central limit theorems for dependent, heterogeneous
  arrays with applications to tail index and tail dependence estimation.
\newblock \emph{Journal of Statistical Planning and Inference}, \textbf{139},
  2091--2110.

\bibitem[Hoga, 2018]{Hoga:2018}
\textsc{Hoga, Y.} (2018).
\newblock A structural break test for extremal dependence in $\beta$-mixing
  random vectors.
\newblock \emph{Biometrika}, \textbf{105}, 627--643.

\bibitem[Hsing, 1986]{hsing1986extreme}
\textsc{Hsing, T.} (1986).
\newblock Extreme value theory for suprema of random variables with regularly
  varying tail probabilities.
\newblock \emph{Stochastic processes and their applications}, \textbf{22},
  51--57.

\bibitem[Joe, 1993]{Joe:1993}
\textsc{Joe, H.} (1993).
\newblock Parametric families of multivariate distributions with given margins.
\newblock \emph{Journal of Multivariate Analysis}, \textbf{46}, 262--282.

\bibitem[Koenker, 2005]{Koenker:2005}
\textsc{Koenker, R.} (2005).
\newblock \emph{Quantile Regression}.
\newblock Cambridge University Press, Cambridge.

\bibitem[Koenker and Bassett, 1978]{Koenker:Bassett:1978}
\textsc{Koenker, R. and Bassett, G.~J.} (1978).
\newblock Regression quantiles.
\newblock \emph{Econometrica}, \textbf{46}, 33--50.

\bibitem[Kokoszka and Taqqu, 1995]{kokoszka1995fractional}
\textsc{Kokoszka, P.~S. and Taqqu, M.~S.} (1995).
\newblock Fractional arima with stable innovations.
\newblock \emph{Stochastic processes and their applications}, \textbf{60},
  19--47.

\bibitem[Kulik and Soulier, 2020]{kulik2020heavy}
\textsc{Kulik, R. and Soulier, P.} (2020).
\newblock \emph{{H}eavy-{T}ailed {T}ime {S}eries}.
\newblock Springer.

\bibitem[Leadbetter et~al., 1983]{Leadbetter:Lindgren:Rootzen:1983}
\textsc{Leadbetter, M.~R., Lindgren, G. and Rootz\'{e}n, H.} (1983).
\newblock \emph{Extremes and Related Properties of Random Sequences and
  Processes}.
\newblock Springer, Berlin.

\bibitem[Ledford and Tawn, 1996]{Ledford:Tawn:1996}
\textsc{Ledford, A.~W. and Tawn, J.~A.} (1996).
\newblock Statistics for near independence in multivariate extreme values.
\newblock \emph{Biometrika}, \textbf{83}, 169--187.

\bibitem[Linton and Whang, 2007]{Linton:Whang:2007}
\textsc{Linton, O.~B. and Whang, Y.-J.} (2007).
\newblock The quantilogram: With an application to evaluating directional
  predictability.
\newblock \emph{Journal of Econometrics}, \textbf{141}, 250--282.

\bibitem[Liu and Lin, 2009]{liu2009strong}
\textsc{Liu, W. and Lin, Z.} (2009).
\newblock Strong approximation for a class of stationary processes.
\newblock \emph{Stochastic Processes and their Applications}, \textbf{119},
  249--280.

\bibitem[Liu and Wu, 2010]{Liu:Wu:2010}
\textsc{Liu, W. and Wu, W.~B.} (2010).
\newblock Asymptotics of spectral density estimates.
\newblock \emph{Econometric Theory}, \textbf{26}, 1218--1245.

\bibitem[McNeil et~al., 2005]{McNeil:Frey:Embrechts:2005}
\textsc{McNeil, A.~J., Frey, R. and Embrechts, P.} (2005).
\newblock \emph{Quantitative Risk Management--Concept, Techniques and Tools}.
\newblock Princeton University Press, Princeton, NJ.

\bibitem[Mikosch and Zhao, 2014]{Mikosch:Zhao:2014}
\textsc{Mikosch, T. and Zhao, Y.} (2014).
\newblock A fourier analysis of extreme events.
\newblock \emph{Bernoulli}, \textbf{20}, 803--845.

\bibitem[Nolan, 2020]{nolan2020univariate}
\textsc{Nolan, J.~P.} (2020).
\newblock \emph{Univariate stable distributions}.
\newblock Springer.

\bibitem[Pham and Tran, 1985]{pham1985some}
\textsc{Pham, T.~D. and Tran, L.~T.} (1985).
\newblock Some mixing properties of time series models.
\newblock \emph{Stochastic processes and their applications}, \textbf{19},
  297--303.

\bibitem[Poon et~al., 2004]{Poon:Rockinger:Tawn:2004}
\textsc{Poon, S.-H., Rockinger, M. and Tawn, J.} (2004).
\newblock Extreme value dependence in financial markets: diagnostics, models,
  and financial implications.
\newblock \emph{The Review of Financial Studies}, \textbf{17}, 581--610.

\bibitem[Rootz\'{e}n, 2009]{Rootzen:2009}
\textsc{Rootz\'{e}n, H.} (2009).
\newblock Weak convergence of the tail empirical process for dependent
  sequences.
\newblock \emph{Stochastic Processes and their Applications}, \textbf{119},
  468--490.

\bibitem[Rosenblatt, 1956]{Rosenblatt:1956}
\textsc{Rosenblatt, M.} (1956).
\newblock A central limit theorem and a strong mixing condition.
\newblock \emph{Proceedings of the National Academy of Sciences of the United
  States of America}, \textbf{42}, 43--47.

\bibitem[Samorodnitsky, 2016]{samorodnitsky:2016:stochastic}
\textsc{Samorodnitsky, G.} (2016).
\newblock \emph{Stochastic Processes and Long Range Dependence}, vol.~26.
\newblock Springer.

\bibitem[Sibuya, 1960]{Sibuya:1960}
\textsc{Sibuya, M.} (1960).
\newblock Bivariate extreme statistics, {I}.
\newblock \emph{Annals of the Institute of Statistical Mathematics},
  \textbf{11}, 195--210.

\bibitem[Smith and Weissman, 1994]{Smith:Weissman:1994}
\textsc{Smith, R.~L. and Weissman, I.} (1994).
\newblock Estimating the extremal index.
\newblock \emph{Journal of the Royal Statistical Society: Series B (Statistical
  Methodology)}, \textbf{56}, 515--528.

\bibitem[Wu, 2005]{Wu:2005}
\textsc{Wu, W.~B.} (2005).
\newblock Nonlinear system theory: another look at dependence.
\newblock \emph{Proceedings of the National Academy of Sciences of the United
  States of America}, \textbf{102}, 14150--14154.

\bibitem[Wu, 2007]{Wu:2007:SIP}
\textsc{Wu, W.~B.} (2007).
\newblock Strong invariance principles for dependent random variables.
\newblock \emph{The Annals of Probability}, \textbf{35}, 2294--2320.

\bibitem[Zhang, 2013]{Zhang:2013}
\textsc{Zhang, T.} (2013).
\newblock Clustering high-dimensional time series based on parallelism.
\newblock \emph{Journal of the American Statistical Association}, \textbf{108},
  577--588.

\bibitem[Zhang, 2015]{Zhang:2015}
\textsc{Zhang, T.} (2015).
\newblock Semiparametric model building for regression models with time-varying
  parameters.
\newblock \emph{Journal of Econometrics}, \textbf{187}, 189--200.

\bibitem[Zhang, 2021{\natexlab{a}}]{Zhang:2021}
\textsc{Zhang, T.} (2021{\natexlab{a}}).
\newblock High-quantile regression for tail-dependent time series.
\newblock \emph{Biometrika}, \textbf{108}, 113--126.

\bibitem[Zhang, 2022]{Zhang:2022}
\textsc{Zhang, T.} (2022).
\newblock Asymptotics of sample tail autocorrelations for tail dependent time
  series: phase transition and visualization.
\newblock \emph{Biometrika}, \textbf{109}, 521--534.

\bibitem[Zhang and Wu, 2011]{Zhang:Wu:2011}
\textsc{Zhang, T. and Wu, W.~B.} (2011).
\newblock Testing parametric assumptions of trends of a nonstationary time
  series.
\newblock \emph{Biometrika}, \textbf{98}, 599--614.

\bibitem[Zhang, 2005]{Zhang:2005}
\textsc{Zhang, Z.} (2005).
\newblock A new class of tail-dependent time series models and its applications
  in financial time series.
\newblock \emph{Advances in Econometrics}, \textbf{20}, 323--358.

\bibitem[Zhang, 2008]{Zhang:2008}
\textsc{Zhang, Z.} (2008).
\newblock Quotient correlation: A sample based alternative to {Pearson's}
  correlation.
\newblock \emph{The Annals of Statistics}, \textbf{36}, 1007--1030.

\bibitem[Zhang, 2021{\natexlab{b}}]{ZhangZ:2021}
\textsc{Zhang, Z.} (2021{\natexlab{b}}).
\newblock On studying extreme values and systematic risks with nonlinear time
  series models and tail dependence measures.
\newblock \emph{Statistical Theory and Related Fields}, \textbf{5}, 1--25.

\bibitem[Zhou and Wu, 2010]{Zhou:Wu:2010}
\textsc{Zhou, Z. and Wu, W.~B.} (2010).
\newblock Simultaneous inference of linear models with time varying
  coefficients.
\newblock \emph{Journal of the Royal Statistical Society: Series B (Statistical
  Methodology)}, \textbf{72}, 513--531.

\end{thebibliography}

\end{document}